\documentclass{amsart}

\usepackage{amssymb}
\usepackage{amsthm,amsmath}
\usepackage{pifont}
\usepackage[numbers]{natbib}
\usepackage{varioref}
\usepackage{fleqn}
\usepackage{hyperref}
\usepackage{graphicx}
\usepackage[left=2.2cm, right=2.2cm, bottom=2.5cm, top=2.5cm]{geometry}

\usepackage{lineno}

\newcommand{\inv}[1]{\frac{1}{ #1 }}

\newcommand{\R}{\mathbb{R}}

\newcommand{\Ca}[1]{\mathcal{#1}}

\theoremstyle{plain}
\newtheorem{theorem}{Theorem}[section]
\newtheorem{lemma}[theorem]{Lemma}
\newtheorem{prop}[theorem]{Proposition}
\newtheorem{definition}{Definition}

\newtheorem{remark}[theorem]{Remark}

\theoremstyle{definition}

\begin{document}
\title{\small Optimal double control problem for a PDE model of goodwill dynamics}

\author{Dominika Bogusz,
Mariusz G\'{o}rajski}

\address{\noindent Dominika Bogusz,  \url{dominika.bogusz@uni.lodz.pl},  Department of Econometrics, Faculty of Economics and Sociology, University of \L \'{o}d\'{z}, \L \'{o}d\'{z}, Poland. \newline \newline \hspace*{0.3cm} Mariusz G\'{o}rajski,
 \url{mariuszg@math.uni.lodz.pl}, 
Department of Econometrics, Faculty of Economics and Sociology, University of \L \'{o}d\'{z}, \L \'{o}d\'{z}, Poland.}




\begin{abstract}
We propose a new optimal model of product goodwill in a segmented market where the state variable is described by a partial differential equation of the Lotka--Sharp--McKendrick type. 
In order to maximize the sum of discounted profits over a  finite time horizon, we control the advertising efforts which influence the state equation and the boundary condition. 
Moreover, we introduce the mathematical representation of consumer recommendations in a segmented market. Based on the semigroup approach, we prove the existence and uniqueness of  optimal controls. Using a  maximum principle, we construct a numerical algorithm to find the optimal solution. Finally, we examine several simulations on  the optimal goodwill model and discover two types of advertising strategies.
\end{abstract}


\maketitle




\section{Introduction}
The subject matter of the present paper is an optimal boundary control problem of product goodwill with double controls. The state variable $G$ is described  by a partial differential equation of Lotka--Sharp--McKendrick type \footnote{Known also as the von Foerster equation} of the form
\begin{align}\label{1}\left\{\begin{array}{lc}
\frac{\partial G(t,a)}{\partial t}+\frac{\partial G(t,a)}{\partial a}+\delta(a) G(t,a)=u^\rho(t,a)\medskip & (t,a)\in\left[0,T\right]\times\left[0,1\right],\\
G(t,0)=\int_0^{1}\left(R(a)G(t,a)+u^\rho(t,a)\right)da+u^\rho_0(t)& t\in\left[0,T\right], \\
G(0,a)=G_0(a) & a\in\left[0,1\right].\end{array}\right. 
\end{align}
Here, $G(t,a)\geq 0$ is the product goodwill at time $t$ for consumer segment $a$, where $a\in [0,1)$ equates with the consumer usage experience, $u(t,a),u_0(t)\geq 0$ are the advertising efforts at time $t$ directed to consumer segment $a$ and to new consumers, respectively, $\rho\in(0,1]$, $R(a)\geq 0$ is  the rate of consumer recommendation for consumers with usage experience $a$ and $\delta(a)\geq 0$ is the depreciation rate of the product goodwill in consumer segment $a$.  We shall present the assumptions about \eqref{1} in Section \ref{s:Goodwill equation}. 
Our  aim  is to choose advertising strategies $u_0$ and $u$ that maximize the sum of the discounted profits in the horizon $T>0$,
 \begin{align}\label{J}
J(G,u_0,u)=\int_0^{1}\int_0^{T}e^{-r t}\left(K\cdot G^{\gamma}(t,a)-\frac{\beta}{2}(u^2(t,a)+u_0^2(t))-c_f\right)dtda, \quad r, K,\beta, \gamma >0
\end{align}
over all admissible controls $(u_0,u)\in U_{0,ad}\times U_{ad}\subset L^2(0,T)\times L^2((0,1)\times(0,T))$ and  subject to the state equation \eqref{1}.

We examine the evolution of the product goodwill in a market divided into segments by the consumer experience using the product.   This usage experience reflects a consumer's perceptions, responses, attitudes, and emotions about using a particular product and has a strong influence over purchasing decisions. For these reasons it is commonly used by companies in creating consumer targeted  offers. 
As far as we know, this type of market segmentation has not previously been included in goodwill models.

  Moreover, since the empirical studies summarized by Bagwell in \cite{bagwell2007})  indicate the existence of decreasing returns to advertising efforts, we include this observation in the new model and we assume a non-linear relation between advertising and goodwill. A similar assumption was used by Weber in \cite{weber2005} and by Mosca and Viscolani in \cite{mosca2004} in a goodwill model expressed by an ordinary differential equation. In addition, we have expanded the existing models by allowing the depreciation of goodwill to be  non-constant, but rather heterogeneous with respect to the usage experience of the product. The main difference between the existing models and that presented in this paper is the process of building the goodwill among consumers with no usage experience. We assume that  goodwill on the part of new consumers depends on advertising directed exclusively to this segment, and by consumer recommendations which can be amplified by advertising aimed at consumers with some experience. Consumer recommendations are considered by most consumers as the most trusted source of information about products. Therefore, they are taken into account in modelling the sales of many products (for example \cite{monahan1984}) but so far, as far as we know, have not yet been taken into account in models describing the dynamics of goodwill involving firms operating in a segmented market.  Our idea of using consumer recommendations in modelling goodwill is based on the empirical evidence, see, for example, \cite{bruce2012} and \cite{agliari2010}, in which the authors claim that consumer recommendations have a strong influence on the level of goodwill.

A general class of optimal control models with heterogeneous state variables which include age structured systems is introduced in \cite{veliov2008} 
and the existence and uniqueness of an optimal solution is proved. 
In our  goodwill equation \eqref{1}, the dependence on the controls is not Lipschitz continuous, hence the existence result from \cite{veliov2008} can not be applied directly. In the Lasiecka and Triggiani monograph \cite{lasiecka2000control} a substantial presentation of the control theory for  the hyperbolic evolution equations with quadratic cost functionals is given.
Following the semigroup approach (see also \cite{Pazy1983, daprato1994}) we prove the existence and uniqueness of an optimal solution to \eqref{J}--\eqref{1} (see Theorem \ref{thm5}).  In  Theorem \ref{thm3} shows that the semigroup-based generalised mild solution to \eqref{1} (see Definition \eqref{df3}) satisfies the definition of solution on the characteristic lines from \cite[Definition 1]{Feichtinger2003}. Hence we are able to use the maximum principle from \cite{Feichtinger2003} to construct a numerical solution to the optimal control problem.

The remainder of this paper is organized as follows. Section \ref{s:lit} briefly reviews the literature on the economic applications of optimal control problems with a state equation described by a first-order hyperbolic partial differential equation.  Section \ref{s:Goodwill equation} presents the new model of product goodwill discusses the economic background of the new idea of market segmentation based on usage experience, giving a mathematical description of consumer recommendations in a segmented market. Section \ref{s:sol} proves the existence and uniqueness of a generalised mild solution to \eqref{1}. Section \ref{s:optimal_sol} establishes the existence and uniqueness of an optimal solution to the goodwill model and presents the necessary optimality  conditions. Section \ref{s:simulation} presents the results of  simulations of the optimal goodwill model obtained by means of a numerical method from Section \ref{s:numerical} .

\section{Literature review}
\label{s:lit}
The Lotka--Sharpe--McKendrick equation provides a framework for the mathematical modelling of many real world phenomena. The most popular application of the equation is a description of age-structured population dynamics   with a boundary condition describing the reproduction process of the population. Population dynamics with an appropriate goal functional is of interest for many biological issues, such as harvesting and birth control (see \cite{Chan1989, daprato1994, park1998optimal, anita2000analysis} and references therein). 

 
Optimal boundary control problems for hyperbolic systems are often used to describe phenomena in the economic and social sciences. An example might be a model for drug initiation including the age distribution of the drug users \cite{almeder2004} or the capital accumulation process in a vintage-capital framework \cite{barucci2001, feichtinger2006}. In recent years, this type of model has been employed in marketing science because researchers have recognized the increasing importance of market segmentation. It has been emphasized that market segmentation strategies improve a company's competitive position and allow better serving the needs of the customers (see \cite{Jha2009}). Moreover, marketing tools, such as advertising, that take into account the specificity of the target group in a particular market segment are more efficient and may also increase the enterprise's profits (see \cite{Mcdonald2004}).

Market segmentation is also applied  to the study of the concept of goodwill, something that has become more and more important in modern business management. Goodwill refers to the difference between the price paid by the buyer for the company and the book value of the assets of that company. This value may be created by the positive experiences of its clients, and may be improved by investment in advertising and other marketing tools. Thus goodwill  translates to an enhancement in the competitiveness of the company and to the acquisition of  future earning power \cite{canibano2000}. Many times it can be observed that a company making a loss is bought at a high price because of its well-known brands. Some real examples of this type of merger and acquisition may be found in \cite[p. 18]{kapferer2012}. Although many researchers have studied this phenomenon, there are still some gaps that prevent a full understanding of the nature of the dynamics of goodwill. 

Modeling is one of way to explore the properties of company goodwill. Nerlove and Arrow in 1962 took the first steps in modeling the concept of goodwill. They interpreted goodwill as the part of the demand for products that is created by current and past advertising efforts (see \cite{Nerlove1962}),  and assumed that the stock of goodwill depreciates over time at a constant rate and depends positively on the advertising effort. They described the dynamics of goodwill in a non-segmented market by an ordinary differential equation. 

The model proposed by Nerlove and Arrow has been modified and analysed by many scientists, who have  recently taken into account market segmentation. They often assume that the firm sells one product in infinitely many segments, indicated by the age of the customers $a$, and the demand in segment $a$ and time $t$ depends on the level $G(t,a)$ of goodwill for this product.  This assumption results in the representation of the goodwill dynamics by a first-order hyperbolic partial differential equation. One example of this approach is \cite{Grosset2005}, who analyse the dynamics of goodwill with a first-order hyperbolic partial differential equation in which a control variable (i.e. advertising efforts) linearly influences the goodwill in the state equation. Newly, Faggian and Grosset in \cite{Faggian2013} reflect the situation in which a firm wants to promote optimally and sells a single product in an age-segmented market, over an infinite time horizon.  In that model, the influence of advertising on goodwill takes a similar form as in the previously mentioned paper.  The same state equation but with a different interpretation  is proposed by Barucci and Gozzi \cite{Barucci1999}. They consider also a goodwill model with market segmentation and describe a  monopolistic firm selling infinitely many products with new goods continuously launched onto the market. In that paper, the control variables represent advertising rates and they appear in the state equation and the boundary condition. However, the boundary condition does not depend on the goodwill variable.

\section{Optimal goodwill model with consumer recommendations}
\label{s:Goodwill equation}

We shall consider a firm in a market with a monopolistic structure divided into segments by the consumers' usage experience $a\in[0,1)$. More precisely,  
the variable $a$ indicates the time spent using the product. The segment $a=0$ includes consumers who have already purchased the product. The maximal usage experience is normalized to the value $1$.  This means that consumers  in segment $a=1$ leave the market forever. The length of the product life cycle is equal to $T$. In each segment $a$ and at each moment of time $t\in[0,T]$ we consider the product goodwill $G(t,a)$ defined the same as in \cite{Nerlove1962}. In order to formalize the concept of consumer recommendation, we assume that $G(t,a)$ is equal to the number of consumers who have been using the product for $a\in[0,1)$ units of time and they continue buying the product at time $t\geq 0$ as the effect of advertising. The firm is able to stimulate different levels of product goodwill by advertising efforts. As we mentioned in the Introduction, the controls $u(t,a)$ and $u_0(t)$  represent the intensiveness of  the advertising efforts at time $t$ directed to consumer segment $a$, and to new consumers, respectively. In our model, we assume a non-linear effect of advertising on goodwill, more precisely, we consider the parameter $\rho$ which reflects the non-linear-concave shape of the advertising response function if $\rho\in(0,1)$ or a linear advertising response function if $\rho=1$. Therefore, $u^\rho(t,a)$ and $u^\rho_0(t)$ positively influence the product goodwill $G(t,a)$ in segment $a$ and the level of product goodwill $G(t,0)$ of new consumers, respectively. Furthermore, there is a natural depreciation rate of goodwill $\delta(a)\geq 0$, different for each consumer segment $a$. This expresses a situation in which the depreciation rate depends on the time spent using the product, and it is natural for an experience product (see \cite{nelson1974}). For this type of goods during the use of the product, consumers  learn about its features and they may update their judgement about it.   This results in changes in the depreciation rate of the goodwill. Therefore, the dynamics of the goodwill are governed by the following PDE:
\begin{align*}
\frac{\partial G(t,a)}{\partial t}+\frac{\partial G(t,a)}{\partial a}+\delta(a) G(t,a)&=u^\rho(t,a)\medskip \quad (t,a)\in\left[0,T\right]\times\left[0,1\right].
\end{align*}

The main novelty in the presented model of goodwill is in the construction of goodwill in the segment of new consumers. A number of empirical studies  of consumer recommendation have concluded that they are a credible source of information  \cite{Brown1987,murray1991}, in particular for consumers without any experience in using the product. Therefore they reduce the risk of purchase decisions and facilitate consumer choice (see \cite{trusov2009}). Hence, the value of goodwill in the segment of new consumers $G(t,0)$ is influenced by the recommendations by consumers with some usage experience. For reasons of clarity, $N(t,a)$ represents consumers who wish to buy the good for the first time at time $t$. This willingness stems from the positive consumer recommendations coming from segment $a$. We distinguish two disjoint groups $N_1(t,a)$ and $N_2(t,a)$ of new consumers affected by recommendations, so that $N(t,a)=N_1(t,a)+N_2(t,a).$ Moreover, by $R(a)$ we denote the relative rate of consumer recommendation in segment $a$  defined  as the ratio $$R(a)=\frac{N_1(t,a)}{G(t,a)}$$ between the number of  new consumers in the first group who want to buy the product influenced by consumers with usage experience $a$, and the total number of consumers in segment $a$.  The consumer recommendations are closely connected with the product quality, which is assumed to be constant and results in that the share $R(a)$ is time homogeneous. On the other hand, usually the quality of the product can only be recognised after some amount of time spent using the product (see \cite{godes2004}), therefore, the rate of consumer recommendation $R(a)$ is heterogeneous with respect to usage experience $a$. Thus the number of new consumers in  the first group is equal to
$$N_1(t,a)=R(a)G(t,a).$$ 

Advertising efforts $u(t,a)$ influence not only the level of goodwill $G(t,a)$ but also the strength of consumer recommendations in segment $a$ by reminding consumers of the reasons for a positive judgement of the product, and thus encouraging them to share their opinion about the product with potential consumers \cite{keller2009}. In conclusion,  advertising efforts $u(t,a)$ act as a reinforcement of the effectiveness of consumer recommendations in  segment $a$. As a result, a new group  of people $N_2(t,a)$ buy the product, which can be calculated by 
$$ N_2(t,a)=\frac{u^\rho(t,a)}{G(t,a)}G(t,a)=u^\rho(t,a),$$
where $\frac{u^\rho(t,a)}{G(t,a)}\eqsim\frac{\Delta G(t,a)}{G(t,a)}$ is the rate of advertising effectiveness in consumer generation $a$.

Finally, we obtain that the  number of new consumers who buy the product at time $t$ as a result of consumer recommendations is equal to
\begin{align*}
\int_0^1 N(t,a)da=\int_0^1(N_1(t,a)+N_2(t,a))da=\int_0^1(R(a)G(t,a)+u^\rho(t,a))da.
\end{align*}
The value of goodwill $G(t,0)$ in the segment of new consumers is also affected by an advertising campaign $u_0(t)$ directed at consumers without any usage experience. Hence, adding the effect of consumer recommendations and advertising effort, we obtain
\begin{align}
\label{2}
G(t,0)=\int_0^{1}\left(R(a)G(t,a)+u^\rho(t,a)\right)da+u^\rho_0(t).
\end{align}
From the above considerations, we obtain  the dynamics of goodwill given by \eqref{1}. 
 
For the maximal advertising intensity (possibly infinite) $I\in(0,\infty]$ denote the sets of admissible controls by $$U_{ad}=\{u\in L^{\infty}((0,T)\times(0,1)): 0 \leq u(t,a)\leq I\ \textrm{ for a.e. }(t,a)\in [0,T]\times[0,1]\},$$ and $$U_{0,ad}=\{u_0\in L^{\infty}(0,T): 0 \leq u_0(t)\leq I\ \textrm{ for a.e. } t\in [0,T]\}.$$

In order to fully describe the optimal control problem, we now define a goal functional. For this purpose, consider a firm  acting  in a monopolistic market who wants to choose one advertising strategy from the admissible strategies to maximize the sum of discounted profits in the horizon $T$ of the form
\begin{align}
J(G,u_0,u)&=\int_0^{1}\int_0^{T}e^{-r t}\left(K\cdot G^{\gamma}(t,a)-\frac{\beta}{2}(u^2(t,a)+u_0^2(t))-c_f(t)\right)dtda
\end{align}
for $(u_0,u)\in U_{0,ad} \times U_{ad}$, where $K, c_f(t)>0$, $\gamma\in(0,1]$, and $r>0$ is  the rate of interest.

Throughout the paper we assume that
$R\colon [0,1]\rightarrow [0,\infty)$ belongs to $ L^\infty(0,1)$, $\delta\colon [0,1]\rightarrow [0,1]$ is a measurable function such that
\begin{align}\label{expstability}
\int_0^{1}R(a)e^{-\int_0^{a}\delta(s)ds}da<1.
\end{align}

\begin{definition}\label{d:optimalsol}
The triple $(G^*,u_0^*,u^*)$ is an optimal solution to the problem of maximizing \eqref{J} subject to \eqref{1} if  $G^*$ is  a generalised mild solution (see Definition \ref{df3}) to \eqref{1} with $(u_0^*,u^*)\in U_{0,ad}\times U_{ad}$  and $$J(G^*,u_0^*,u^*)\geq J(G,u_0,u)$$ holds for any admissible controls $(u_0,u)\in U_{0,ad}\times U_{ad}$ and $G$ satisfying \eqref{1}.
\end{definition}
In the next section, we will prove that for any $(u_0,u)\in U_{0,ad}\times U_{ad}$ there exists a generalised mild solution to \eqref{1} (Theorem \ref{thm2} in Section \ref{s:sol}).

\section{Existence and uniqueness of the solution to the goodwill equation}
 \label{s:sol}
In order to  prove existence and uniqueness of  \eqref{1}  we use the semigroup approach and the notion of a generalised mild solution to \eqref{1}.

\subsection{Reformulating the goodwill equation to first order  system with homogeneous boundary conditions}

Now, we transform the problem \eqref{1} (similarly to \cite{daprato1994}) into a problem with homogeneous boundary conditions \eqref{eqQ}. First, we denote by
\begin{align}\label{w}
w(t)=\int_0^{1}u^\rho(t,a)da+u^\rho_0(t),\quad t\in[0,T]
\end{align}
the controls from the boundary condition of \eqref{1}.
\begin{remark}
If $(u_0,u)\in U_{0, ad}\times U_{ad}$, then the function $w\colon [0,T]\rightarrow[0,\infty)$ belongs to $L^{\infty}(0,T)$.
\end{remark}
Denote
\begin{align}\label{D}
D(a)=e^{-\int_0^{a}\delta(s)ds},\quad a\in[0,1]
\end{align}
the future value in time $a$ of $1$ unit of goodwill in segment of new consumers. Moreover, let
\begin{align} \label{Q}
Q(t,a)=G(t,a)-g(t,a),\quad (t,a)\in[0,T]\times[0,1] \text{ a.e.},
\end{align}
where for any fixed $t\in[0,T]$ the function $g(t,\cdot)$ is a solution to the equation
\begin{align}\label{g}
\left\{\begin{array}{ll}
\frac{\partial g(t,a)}{\partial a}=-\delta(a)g(t,a) & (t,a)\in[0,T]\times[0,1],\\
g(t,0)=\mu w(t),
\end{array}\right.
\end{align}
where \begin{align}\label{mu}\mu=\frac{1}{1-\int_0^{1}R(a)D(a)da}>1
\end{align}
by \eqref{expstability}. Then, the solution of \eqref{g} has the following form
\begin{align}\label{sol_g}
 g(t,a)=\mu w(t)D(a).
\end{align}

\begin{theorem}\label{thm1} Assume \eqref{expstability} holds and $u\in U_{ad}$, $u_0\in U_{0,ad}$ are positive-valued, continuous functions such that $\frac{\partial u}{\partial t}$ and $\frac{d u_0}{dt}$ are  continuous. Moreover, let $\delta$ be a continuous and equations \eqref{Q} and \eqref{g} hold. Then,
$G$ is a classical solution \footnote{A function $G\colon [0,T]\times[0,1]\to \mathbb{R}$ is called a classical solution to \eqref{1} if $G\in C^1\left([0,T]\times[0,1]\right)$ and satisfies the equation \eqref{1} for all $(t,a)\in[0,T]\times[0,1]$.} to \eqref{1} if and only if $Q\in C^1\left([0,T]\times[0,1]\right)$ satisfies the following equation
\begin{align}\label{eqQ}
\left\{\begin{array}{ll}
\frac{\partial Q(t,a)}{\partial t}+\frac{\partial Q(t,a)}{\partial a}=-\delta(a)Q(t,a)+u^\rho(t,a)-\mu D(a)w'(t) & (t,a)\in[0,T]\times[0,1],\\
Q(t,0)=\int_0^{1}R(a)Q(t,a)da & t\in[0,T],\\
Q(0,a)=G_0(a)-\mu w(0)D(a) & a\in[0,1].
\end{array}\right.\end{align}
\end{theorem}
\proof
Let $G$ be a classical solution to \eqref{1}. Then, the assumptions on controls $u, u_0$ guarantee that there exists continuous derivative $w'$. By  \eqref{Q}-\eqref{g} and \eqref{1}  we obtain
\begin{align*}
\frac{\partial Q(t,a)}{\partial t}+\frac{\partial Q(t,a)}{\partial a}&=\frac{\partial G(t,a)}{\partial t}+\frac{\partial G(t,a)}{\partial a}-\frac{\partial g(t,a)}{\partial t}-\frac{\partial g(t,a)}{\partial a}\\=&-\delta(a)G(t,a)+u^\rho(t,a)-\mu D(a)w'(t)+\delta(a)g(t,a)\\ &=-\delta(a)Q(t,a)+u^\rho(t,a)-\mu D(a)w'(t)
\end{align*}
for all $(t,a)\in [0,T]\times[0,1]$.
Moreover, for all $a\in [0,1]$ from \eqref{sol_g} we have the initial condition  
\begin{equation*}
Q(0,a)=G(0,a)-g(0,a)=G_0(a)-\mu w(0)D(a).
\end{equation*}
By \eqref{mu}-\eqref{sol_g} and \eqref{2}  we have the boundary condition 
\begin{align*}
Q(t,0)&=G(t,0)-g(t,0)=\int_0^{1}R(a)G(t,a)da+w(t)-\mu w(t)\\&=\int_0^{1}R(a)Q(t,a)da+\int_0^{1}R(a)g(t,a)da-\mu w(t)+w(t)\\&=\int_0^{1}R(a)Q(t,a)da+w(t)\left(\mu\left(\int_0^{1}R(a)D(a)da-1\right)+1\right)\\&=\int_0^{1}R(a)Q(t,a)da
\end{align*}
for all $t\in[0,T]$. Similarly, one can prove the "if" implication. 
\endproof

\subsection{The goodwill equation as a homogeneous Cauchy problem in a Hilbert space}
Since we want to consider non-smooth controls $u,u_0$, we need to introduce a weaker concept of solution to \eqref{1}.  Let $L^2(0,1)$ denote the Lebesgue space of square integrable functions on $(0,1)$ and for $p\geq 1$ $W^{1,p }(0,1)$ is the Sobolev space of $p$-integrable functions with weak derivative in $L^p(0,1)$. We rewrite \eqref{eqQ} as an evolution equation in $L^2(0,1)$.
Define a linear unbounded operator on $L^2(0,1)$ by 
\begin{align*}
\mathcal{D}(\mathcal{A})&=\left\{ \phi\in W^{1,2}(0,1)\colon \phi(0)=\int_0^{1}R(a)\phi(a)da\right\},\\
\mathcal{A}\phi&=-\phi'-\delta \phi,\quad \phi\in \mathcal{D}(\mathcal{A}).
\end{align*}
Then, under the assumption \eqref{expstability} the operator  $(\mathcal{A},\mathcal{D}(\mathcal{A}))$ generates strongly continuous semigroup of linear operators $(S(t))_{t\geq 0}$ on $L^2(0,1)$ (see \cite{Webb1985, daprato1994}). The semigroup $\left(S(t)\right)_{t\geq 0}$ is given by
\begin{align}\label{S}
(S(t)\phi)(a)=\left\{\begin{array}{ll}B_{\phi}(t-a)D(a) & a\in[0,t],\\
\phi(a-t)\frac{D(a)}{D(a-t)} & a>t,
\end{array}\right.
\end{align}
where $B_{\phi}$ be a solution of  Volterra integral equation
\begin{align}\label{B}
B_{\phi}(t)=F_{\phi}(t)+\int_0^tK(t-s)B_{\phi}(s)ds,\quad t\geq 0
\end{align}
with
$$K(t)=\left\{\begin{array}{ll}R(t)D(t)&\textrm{ if } t\in[0,1],\\0&\textrm{ if } t>1\end{array}\right.$$
and
$$F_{\phi}(t)=\int_{t\wedge 1}^{1}\phi(s-t)R(s)\frac{D(s)}{D(s-t)}ds .$$
\begin{remark}
By \eqref{expstability} the semigroup $(S(t))_{t\geq 0}$ is uniformly exponentially stable (cf. section 4.5 in Chapter VI in \cite{Engel2006}).
\end{remark}

\begin{prop} \label{prop1} Assume \eqref{expstability}. Then, for all $\phi \in L^2(0,1)$
\begin{enumerate}
\item the equation \eqref{B} possesses the unique continuous solution $B_{\phi}$ on $[0,\infty)$;
\item  the function $B_{\phi}$ satisfies  \begin{align} \label{estB}\|B_{\phi}\|_{L^{\infty}(0,1)}\leq\mu \|\phi\|_{L^{\infty}(0,1)}\|R\|_{L^{\infty}(0,1)}.
\end{align}
\item For $D$ defined in \eqref{D} the function $B_{D}$ satisfies 
\begin{equation*}
B_{D}(t)=\int_{t\wedge1}^{1}R(s)D(s)ds+\int_0^{t\wedge1}R(s)D(s)B_{D}(t-s)ds
\end{equation*}
and is differentiable on $[0,\infty)$.
Moreover, $B'_{D}\in L_{loc}^\infty(0,\infty)$ is the solution to
\begin{align}\label{dB}
B'_{D}(t)=\left\{\begin{array}{ll} -\frac{1}{\mu}R(t)D(t)+\int_0^t R(s)D(s)B'_{D}(t-s)ds & t\in[0,1),\\
\int_0^{1}R(s)D(s)B'_{D}(t-s)ds & t\geq 1
\end{array}\right.
\end{align}
and hence for all $t>0$ satisfies
\begin{align}\label{dB1}
\| B'_{D}\|_{L^\infty(0,t)}\leq \mu \|R\|_{L^\infty(0,1)}. 
\end{align}

\end{enumerate}
\end{prop}
\begin{proof}
Let $\phi\in L^2(0,1)$.
\begin{enumerate}
\item[Part 1.] Notice that $F_\phi \in C([0, \infty))$, thus the result follows from Theorem 3.5 in Chapter 2 in  \cite{gripenberg1990}.

\item[Part 2.] Since $\frac{D(s)}{D(s-t)}=e^{-\int_{s-t}^s\delta(u)du}<1$ for all $s\geq t\geq 0$, from \eqref{B} we obtain
\begin{align*}
\|B_{\phi}(\cdot)\|_{L^{\infty}(0,1)}\leq\|\phi\|_{L^{\infty}(0,1)}\|R\|_{L^{\infty}(0,1)}+\int_0^{1}R(s)D(s)ds \|B_{\phi}\|_{L^{\infty}(0,1)}.
\end{align*}
Hence and using the definition of $\mu$ in \eqref{mu} we get
\begin{align*}
\|B_{\phi}(\cdot)\|_{L^{\infty}(0,1)}\leq \mu\|\phi\|_{L^{\infty}(0,1)}\|R\|_{L^{\infty}(0,1)}.
\end{align*}

\item[Part 3.]  Since $F_D(t)=\int_{t\wedge 1}^{1}R(s)D(s)ds, t\geq 0$ is differentiable a.e on $[0,\infty)$ and $F'\in L^\infty(0,\infty)$, by Theorem 3.3 from Chapter 3 in \cite{gripenberg1990} the  $B_D$ is differentiable a.s on $[0,\infty)$  the derivative $B'_D\in L_{loc}^\infty(0,\infty)$ satisfies \eqref{dB}. \eqref{dB1} is a simple consequence of \eqref{dB}.
\end{enumerate}
\end{proof}

Notice that the equation \eqref{eqQ} can be reformulated as a Cauchy problem  in the Hilbert space $L^2(0,1)$:
\begin{align}\label{eqQ2}
\left\{\begin{array}{ll}
Q'(t)=\mathcal{A}Q(t)+u^\rho(t)-\mu D w'(t)\quad t\in[0,T],\\
Q(0)=G_0-\mu w(0)D.&\end{array}
\right.
\end{align}
Based on \cite{Pazy1983} we introduced the following definition
\begin{definition}\label{df1} A measurable function $Q\colon [0,T]\rightarrow L^2(0,1)$ is called a mild solution to \eqref{eqQ2} if $G_0\in L^{2}(0,1)$ and $w\in W^{1,1}([0,T])$ and for any $t\in[0,T]$ one has 
\begin{align}\label{mildQ}
Q(t)=S(t)Q(0)+\int_0^tS(t-s) u^\rho(s)ds-\mu\int_0^tS(t-s)D w'(s)ds.
\end{align}
\end{definition}
\begin{remark}\label{r:2}
Let the assumptions of Theorem \ref{thm1} be satisfied. If $Q$ is a classical solution to \eqref{eqQ}, then $Q$ is a mild solution to \eqref{eqQ2} (cf. \cite{Engel2000} and \cite{Pazy1983}).
\end{remark}

\begin{definition}
\label{df3} A measurable function $G\colon [0,T]\rightarrow L^2(0,1)$ is called a generalised mild solution to \eqref{1} if $G_0\in L^{2}(0,1)$, $(u_0,u)\in L^2(0,T)\times L^2((0,T)\times (0,1))$ and for any $t\in[0,T]$ one has 
$$G(t)=S(t)G_0+\int_0^tS(t-s)u^\rho(s)ds-\mu\mathcal{A}\int_0^tS(t-s)D w(s)ds.$$
We write $G=G(u,u_0;G_0)$ to denote that generalised mild solution to \eqref{1} depends on the controls $u$, $u_0$ and the initial value $G_0$. 
\end{definition}
\begin{remark} \label{r:3}
If $Q$ is a mild solution to \eqref{eqQ2}, then $G=Q+g$ is a generalised mild solution to \eqref{1}.
\end{remark}
\begin{proof}[Proof of Remark \ref{r:3}]
Indeed, from \eqref{mildQ},   \eqref{sol_g} and \eqref{Q} we get 
\begin{align*}
G(t)=S(t)G_0-\mu w(0)S(t)D +\mu w(t)D+\int_0^t S(t-s) u^\rho(s)ds-\mu\int_0^t  w'(s) S(t-s)Dds.
\end{align*}
We integrate by parts the last term to obtain 
\begin{align*}
G(t)=&S(t)G_0-\mu w(0)S(t)D +\mu w(t)  D +\int_0^tS(t-s) u^\rho(s)ds-\mu\left(\left[w(s)S(t-s)D \right]_0^t+\mathcal{A}\int_0^t w(s)S(t-s)D ds\right)\\
&=S(t)G_0+\int_0^tS(t-s) u^\rho(s)ds-\mu\mathcal{A}\int_0^t w(s)S(t-s)D ds.
\end{align*}
\end{proof}
From Remarks \ref{r:2}, \ref{r:3} and Theorem \ref{thm1} we obtain:
\begin{remark}
Under the assumptions of Theorem \ref{thm1} if $G$ is a classical solution to \eqref{1}, then $G$ is a generalised mild solution to \eqref{1}.
\end{remark}

Here and subsequently, $z$ stands for \begin{align} 
\label{z} z(t)=\int_0^t w(s)S(t-s)D ds,
\end{align}
where $w$ is given by \eqref{w}.
\begin{theorem}\label{thm2}
Let $G_0\in L^2(0,1)$, $(u_0,u)\in L^2(0,T)\times L^2((0,T)\times(0,1))$ and \eqref{expstability} holds. Then,
\begin{enumerate}
\item there exists a unique generalised mild solution $G(u,u_0;G_0)$ to \eqref{1} i.e. $z(t)\in \mathcal{D}(\mathcal{A})$ for all $t\in [0,T]$,
\item $\mathcal{A}z\in C(0,T;L^2(0,1))\cap L^{\infty}(0,T; L^{\infty}(0,1))$.
\end{enumerate}
Furthermore, if $(u_0^1,u^1)\in L^2(0,T)\times L^2((0,T)\times(0,1))$ and $G=G(u,u_0;G_0)$ and $G_1=G_1(u^1,u_0^1;G_0)$ are the generalised mild solution to \eqref{1}, then there exist $L_1,L_2>0$ such that
\begin{enumerate}
\item[3.]  \begin{align*}
\sup_{t\in[0,T]} \|G(t)\|_{L^2(0,1)}&\leq L_1\left(\|G_0\|_{L^2(0,1)}+\|u\|^\rho_{L^2((0,1)\times(0,T))}+\|u_0\|^\rho_{L^2(0,T)}\right),\\
\sup_{t\in[0,T]} \|G(t)-G_1(t)\|_{L^2(0,1)}&\leq L_1\left(\|u-u^1\|^\rho_{L^2((0,1)\times(0,T))}+\|u_0-u^1_0\|^\rho_{L^2(0,T)}\right),
\end{align*}

\item[4.] if moreover  $(u_0,u), (u^1_0,u^1)\in  U_{0,ad} \times U_{ad} $ and $G_0\in L^{\infty}(0,1)$, then we have
\begin{align*}
\sup_{t\in[0,T]} \|G(t)\|_{L^{\infty}(0,1)}&\leq L_2\left(\|G_0\|_{L^\infty(0,1)}+ \|u\|^\rho_{L^{\infty}((0,1)\times(0,T))}+\|u_0\|^\rho_{L^{\infty}(0,T)}\right),\\
\sup_{t\in[0,T]} \|G(t)-G_1(t)\|_{L^\infty(0,1)}&\leq L_2\left(\|u-u^1\|^\rho_{L^\infty((0,1)\times(0,T))}+\|u_0-u^1_0\|^\rho_{L^\infty(0,T)}\right).
\end{align*}
\end{enumerate}
\end{theorem}
\begin{proof}
Let  $z^1(t)=\int_0^tS(t-s)D w^1(s)ds$ where $w^1(t)=\int_0^1(u^1(t,a))^\rho da+(u_0^1(t))^\rho$ for all $t\geq 0$. By \cite[Theorem 2.2]{daprato1994} we obtain that $z(t), z^1(t)\in \mathcal{D}(\mathcal{A})$  for all $t\in[0,T]$, and  $\mathcal{A}z(\cdot),\mathcal{A}z^1(\cdot)\in C([0,T];L^2(0,1))$ and \begin{align}\label{23}\|\mathcal{A}(z(t)-z^1(t))\|_{L^2(0,1)}\leq L\|w-w^1\|_{L^2(0,T)}\end{align}
for some $L>0$.
Moreover, by H\"older's continuity of $x\mapsto x^\rho$  and the H\"older inequalities we have 
\begin{align}\label{22}
\|w-w^1\|_{L^2(0,T)}&\leq \left\|t\mapsto\int_0^{1}(u^\rho(t,a)-(u^1(t,a))^\rho)da\right\|_{L^2(0,T)}+\left\|u_0^\rho-(u_0^1)^\rho\right\|_{L^2(0,T)}\\&\leq
\left\|t\mapsto\int_0^{1}|u(t,a)-u^1(t,a)|^\rho da\right\|_{L^2(0,T)}+\left\|(u_0-u_0^1)^\rho\right\|_{L^2(0,T)}\notag\\
&\leq T^{\frac{1-\rho}{2}}\left(\|u-u^1\|^\rho_{L^2\left((0,1)\times(0,T)\right)}+\|u_0-u_0^1\|^\rho_{L^2(0,T)}\right).\notag
\end{align}
Similarly, we have

\begin{align}
\label{21} \left\|\int_0^tS(t-s) (u^\rho(s)-(u^1(s))^\rho)ds\right\|_{L^2(0,1)}&\leq M(T)\int_0^t\left\| u^\rho(s)-(u^1(s))^\rho\right\|_{L^2(0,1)}ds
\\&\leq M(T)\int_0^t\left\| (u(s)-u^1(s))^\rho\right\|_{L^2(0,1)}ds\notag\\&\leq T^{1-\frac{\rho}{2}}M(T)\|u-u^1\|^\rho_{L^2\left((0,1)\times(0,T)\right)},\notag
\end{align}
where $M(T)=\sup_{t\in[0,T]}\|S(t)\|_{\mathcal{L}(L^2(0,1))}<\infty$.
Therefore, by \eqref{23}-\eqref{21} we obtain inequalities from part 3. of Theorem \ref{thm2}, where
 $L_1=\max\left\{M(T),T^{1-\frac{\rho}{2}}M(T),\mu L T^{\frac{1-\rho}{2}}\right\}$.

Now we prove part 2. and the inequalities from part 4. of Theorem \ref{thm2}. 
By the definition of semigroup $(S(t))_{t\geq 0}$ (cf. \eqref{S})  we can rewrite \eqref{z} as
\begin{align}\label{z2}
z(t,a)=D(a)\left(\int_0^{t\wedge a}B_{D}(s-a)w(t-s)ds+\int_{t\wedge a}^tw(t-s)ds\right)
\end{align}
for all $a\in [0,1]$ and $t\in[0,T]$.
Differentiating \eqref{z2} with respect to $a$ we obtain
\begin{align}\label{Az}
\mathcal{A}z(t)(a)=\left\{\begin{array}{ll}
-D(a)\left[\frac{1}{\mu}w(t-a)+\int_a^{t}B'_{D}(s-a)w(t-s)ds\right] & t\geq a,\\
0 & t<a
\end{array}\right.
\end{align}
for all $t\in[0,T]$. 
Hence
\begin{align}\label{24}
\|\mathcal{A}z(t)\|_{L^{\infty}(0,1)}&\leq\left(\frac{1}{\mu}+t\left\|B'_D\right\|_{L^\infty(0,T)}\right)\|w\|_{L^{\infty}(0,T)}\\&\leq \left(\frac{1}{\mu}+t\mu \|R\|_{L^\infty(0,1)}\right)  \left(\|u\|^\rho_{L^\infty\left((0,1)\times(0,T)\right)}+\|u_0\|^\rho_{L^\infty(0,T)}\right), \notag
\end{align}
where the last inequality follows form the third part of Proposition \ref{prop1}. Since $0<D(a)\leq 1$ for every $a\in[0, 1]$, by \eqref{estB} we obtain 
\begin{align}\label{27}
\|S(t)\phi\|_{L^{\infty}(0,1)}\leq \left(\mu \|R\|_{L^{\infty}(0,1)}\vee 1\right)\|\phi\|_{L^\infty(0,1)}
\end{align}
for all $\phi\in L^{\infty}(0,1)$.
Similarly, by \eqref{estB} we have 
\begin{align}\label{26}
\left\|\int_0^tS(s) u^\rho(t-s)ds\right\|_{L^\infty(0,1)}\leq t
(\mu \|R\|_{L^{\infty}(0,1)}\vee 1)\|u\|^\rho_{L^{\infty}((0,T)\times(0,1))}.
\end{align}
Finally,  form \eqref{24}-\eqref{26} we get
\begin{align*}
&\|G(t)\|_{L^{\infty}(0,1)}\leq \|S(t)G_0\|_{L^{\infty}(0,1)}+ \left\|\int_0^tS(t-s) u^\rho(s)ds\right\|_{L^{\infty}(0,1)}+ \|\mu \mathcal{A}z(t)\|_{L^{\infty}(0,1)}\\ &\leq L_2\left(\|G_0\|_{L^{\infty}(0,t)}+\|u\|^\rho_{L^\infty\left((0,1)\times(0,T)\right)}+\|u_0\|^\rho_{L^\infty(0,T)}\right)
\end{align*}
for all  $G_0\in L^{\infty}(0,1)$ and $t\in[0,T]$, where $$L_2=\max\left\{1+t\mu^2 \|R\|_{L^\infty(0,1)},\left(\mu \|R\|_{L^{\infty}(0,1)}\vee 1\right),t\left(\mu \|R\|_{L^{\infty}(0,1)}\vee 1\right)\right\}.$$
\end{proof}

\subsection{The relation between the generalised mild solution and the solution along the characteristic lines}
\begin{theorem}\label{thm3}
If \eqref{expstability} holds, then the generalised mild solution to \eqref{1} satisfies the following formulae:
\begin{align} \label{GFeihtinger1}
G(t,t+c)=G_0(t)-\int_0^t \delta(s+c)G(s,s+c)ds+\int_0^t u^\rho(s,s+c)ds,
\end{align}
 for all $c\in(0,1]$ and all $t\in[0,1-c]$, and
\begin{align} \label{GFeihtinger2}
G(t,t+c)=G(-c,0)-\int_{-c}^t \delta(s+c)G(s,s+c)ds+\int_{-c}^tu^\rho(s,s+c)ds,
\end{align}
 for all $c\in[-T,0]$ and all $t\in(-c,T]$, and
\begin{align} \label{GFeihtinger3}
G(t,0)=\int_{0}^1 R(a)G(t,a)da+\int_{0}^1u^\rho(t,a)da+u^\rho_0(t),
\end{align}
 for all $t\in[0,T]$.
\end{theorem}
\begin{proof} Identity \eqref{GFeihtinger3} follows directly from the definition of generalised mild solution of \eqref{1} and \eqref{S}, \eqref{Az}.  
We prove the second equality \eqref{GFeihtinger2}, the formula \eqref{GFeihtinger1} can be proven similarly.

Observe that by \eqref{S}   we obtain \begin{align}
\label{32}
\left(S(t)G_0\right)(t+c)=B_{G_0}(-c)D(t+c)
\end{align}
and
\begin{align}\notag
\left(\int_0^tS(s) u^\rho(t-s)ds\right)(t+c)&\\&=\int_0^{t+c}u^\rho(t-s,t+c-s)\frac{D(t+c)}{D(t+c-s)}ds+\int_{t+c}^tB_{ u^\rho(t-s)}(s-t-c)D(t+c)ds \notag\\&=D(t+c)\left(\int_0^{t+c}u^\rho(r-c,r)\frac{1}{D(r)}dr+\int_{0}^{-c}B_{ u^\rho(-c-r)}(r)dr\right). \label{33}
\end{align}
for a.e. $c\in[-T,0]$ and $t\in(-c,T]$.
Moreover, using \eqref{Az} we have 
\begin{align}\label{34}
-\mu\left(\mathcal{A}\int_0^tw(t-s)S(s)D ds\right)(t+c)= D(t+c)\left(w(-c)-\mu\int_0^{-c}B'_{D}(-c-s)w(s)ds\right).
\end{align}
As a result, by \eqref{32}-\eqref{34} the generalised mild solution of \eqref{1} takes the form
\begin{align}\label{36}
G(t,t+c)&=D(t+c)\Big(B_{G_0}(-c)+\int_0^{t+c}u^\rho(r-c,r)\frac{1}{D(r)}dr+\int_{0}^{-c}B_{ u^\rho(-c-r)}(r)dr\nonumber\\&+w(-c)-\mu\int_0^{-c}B'_{D}(-c-s)w(s)ds\Big)
\end{align}
for a.e. $c\in[-T,0]$ and for all $t\in[-c,T]$.
Hence for all $c\in[-T,0]$ the mapping $t\mapsto G(t,t+c)$ has an absolutely continuous version. 
Finally,  for a.e. $t\in[-c, T]$ the derivative $t\mapsto G(t,t+c)$ is equal to
\begin{align*}
\frac{d G(t,t+c)}{dt}&=-\delta(t+c)D(t+c)\Big(B_{G_0}(-c)+\int_0^{t+c}u^\rho(r-c,r)\frac{1}{D(r)}dr+\int_{0}^{-c}B_{ u^\rho(-c-r)}(r)dr\\&+ w(c)-\mu\int_0^{-c}B'_{D}(-c-s)w(s)ds\Big)+u^\rho(t,t+c)\\&=-\delta(t+c)G(t,t+c)+u^\rho(t,t+c),
\end{align*} 
which prove the equality \eqref{GFeihtinger2}.
\end{proof}
\begin{remark}\label{r:4}
By Theorem \ref{thm2}.4 and Theorem \ref{thm3}, if $G_0\in L^{\infty}(0,1)$, and $(u,u_0)\in U_{0,ad}\times U_{ad}$ then the generalised mild solution $G(u,u_0;G_0)$ of \eqref{1} is the solution of \eqref{1} along the characteristic lines as in \cite{Feichtinger2003}. In particular $G(u,u_0;G_0)$ belongs to $L^{\infty}(0,T;L^{\infty}(0,1))\cap C(0,T;L^2(0,1))$.  
\end{remark}

\section{The optimal solution to the goodwill model}
\label{s:optimal_sol}

\subsection{The existence and uniqueness of an optimal solution}

We prove the existence of an optimal solution to problem \eqref{1}, \eqref{J}  (see Theorem \ref{thm5}) using the classical results for a general extreme problem in a Hilbert space $H$ (see Theorem \ref{thm4}). 

Let $f\colon U\rightarrow \mathbb{R}$ be a functional defined on subset $U\subset H$. Consider the optimization problem
\begin{align}\label{EP}
\inf_{h\in U}f(h). 
\end{align}
  
\begin{theorem}\cite[Theorems 7.3.5, 7.3.7]{kurdila2005}\label{thm4} If the functional $f\colon U\rightarrow \mathbb{R}$ is lower semicontinuous, convex and coercive, and the set $U$ is not empty, closed and convex, then there exist a solution $u^*\in U$ to \eqref{EP} i.e. $f(u^*)=\inf_{h\in H} f(h)$. 
\end{theorem}
Notice that if $U$ is bounded, then in Theorem \ref{thm4} the coercivity of $f$ is superfluous. Moreover, if $f$ is additionally strictly convex, then the solution  $u^*\in U$ is unique.

\begin{theorem}\label{thm5}
Assume that \eqref{expstability} holds and let $G_0\in L^2(0,1)$ be an almost everywhere positive function. The optimal control problem \eqref{1} and \eqref{J} admits a unique solution.
\end{theorem}
In the proof of Theorem \ref{thm5} we need the following lemma. 
\begin{lemma}\label{l:positive}
Let $\Ca{E}_1, \Ca{E}_2$ be Banach function spaces over a $\sigma$-finite
measure space $(S,\Sigma,\mu)$. Consider a operator $F:\Ca{E}_1\to\Ca{E}_2$ such that, $\mu$-almost everywhere, 
\begin{align}\label{concave}
F(\alpha f_1+(1-\alpha)f_2)>\alpha F(f_1)+(1-\alpha)F(f_2),
\end{align}
and let $f^*$ be a positive and strictly concave functional on $\Ca{E}_2$. Then,
the composition $f^*\circ F$ is strictly concave functional on $\Ca{E}_1$.
\end{lemma}
\proof
Let $\alpha\in(0,1)$ and $f_1,f_2\in\Ca{E}_1$. Then, from the assumptions on $f^*$ and \eqref{concave} it follows that
\begin{align*}
(f^*\circ F)(\alpha f_1+(1-\alpha)f_2)>f^*(\alpha F(f_1)+(1-\alpha)F(f_2))>\alpha f^*( F(f_1))+(1-\alpha)f^*(F(f_2)),
\end{align*}
where in the last inequality we use strict concavity of $f^*$. 
\endproof

\proof[Proof of Theorem \ref{thm5}.]
Fix $G_0\in L^2(0,1)$ such that $G(a)> 0$ for a.e $a\in(0,1)$. We can rewrite the problem \eqref{1}, \eqref{J} as \eqref{EP} with  $H= L^2(0,T)\times L^2((0,T)\times(0,1))$ and the functional $f(u_0,u)=-J(G,u_0,u)$  defined on $U=U_{0,ad}\times U_{ad}\subset H$. The set of admissible controls $U$ is obviously nonempty and convex. To prove closedness of $U$ let $\{(u_{0,n},u_n)\}_{n\geq 1}$ be a  sequence of admissible controls converging in $H$-norm to $(u_0,u)$. We show that $u_0\in U_{0,ad}$ and $u\in U_{ad}$. Indeed, the sets $A=\{t\in(0,T):u_0(t)<0\}$ $B=\{t\in(0,T): u_0(t)-I>0\}$ are measurable and $\int_Au_{0,n}(t)dt\geq 0$, $\int_B(u_{0,n}(t)-I)dt\leq 0$ for all $n\geq 1$. Taking the limits in these two sequences of integrals we obtain $\int_Au_{0}(t)dt\geq 0$ and $\int_B(u_{0}(t)-I)dt\leq 0$. Thus $|A|=0$ and $|B|=0$ and we conclude that $0 \leq u_{0}(t)\leq I$ for a.e. $t\in(0,T)$. Hence $u_{0}\in U_{0,ad}$. The same argument can be used to prove that $u\in U_{ad}$. 

 Moreover, since $\rho,\gamma\in(0,1]$ one can prove that $J$ given by \eqref{J} is strictly concave. Indeed, first notice that $J$ can be represented as follows
\begin{align}
J(G,u_0,u)&=J_1(u_0,u)-\frac{\beta}{2}\int_0^1\int_0^Te^{-rt}(u^2(t,a)+u_0^2(t)-c_f)dtda\\
J_1(u_0,u)&=\int_0^1\int_0^Te^{-rt}K_\Pi\left(G_1(u^\rho)(t,a)+G_2(u_0^\rho)(t,a)\right)^\gamma dtda, 
\end{align}
for all $(u_0,u)\in H$ and $G(u_0,u;G_0)=G_1(u^\rho;G_0)+G_2(u^\rho_0)$ satisfying \eqref{1}, where $G_1:L^2((0,1)\times(0,T))\to C([0,T];L^2(0,1))$, $G_1(u;G_0)(t,a)=(S(t)G_0)(a)+(\int_0^tS(t-s)u(s)ds)(a)-\mu(\Ca{A}\int_0^tS(t-s)Dw_u(s)ds)(a)$, $w_u(s)=\int_0^1u(s,a)da$ and $G_2:L^2(0,T)\to C([0,T];L^2(0,1))$, $G_2(u_0)(t,a)=-\mu(\Ca{A}\int_0^tS(t-s)Du_0(s)ds)(a)$. 
Since any norm in a Hilbert space is strictly convex, the mapping $H\ni(u,v)\mapsto-\frac{\beta}{2}\int_0^1\int_0^Te^{-rt}(u^2(t,a)+v^2(t)-c_f)dtda$ is strictly concave. Hence it is enough to show that $J_1:H\to\R$ is strictly concave.    
We notice that $J_1$ is a composition of  the Niemycki operator $(u_0,u)\mapsto (F(u_0),F(u))$ on $H$  with the strictly concave function $F(x)=x^\rho$ and the positive and strictly concave functional $$(u_0,u)\mapsto \int_0^1\int_0^Te^{-rt}K_\Pi\left(G_1(u)(t,a)+G_2(u_0)(t,a)\right)^\gamma dtda$$ on $H$. Positivity of last functional follows by assumption $G_0 >0$ and by positivity of the Lotka-Sharp-McKendrick semigroup $(S(t))_{t\geq 0}$ (cf. \cite[Section 4 in Chapter IV ]{Engel2000}). Hence by Lemma \ref{l:positive}  $J_1$ is strictly concave. 

Furthermore, in the case of $I=\infty$ we show that the functional $f$ is coercive. For this purpose consider a sequence $\{(u_{0,n},u_n)\}_{n\geq 1}\subset U_{0,ad}\times U_{ad}$ such that $\|(u_{0,n},u_n\|_{L^2(0,T)\times L^2((0,T)\times(0,1))}\rightarrow\infty$, thus $\|u_n\|_{L^2((0,T)\times(0,1))}\rightarrow\infty$ or $\|u_{0,n}\|_{L^2(0,T)}\rightarrow\infty$. From Theorem \ref{thm2} for each element of sequence $\{(u_{0,n},u_n)\}_{n\geq1}$   there exists the generalised mild solution $G_n=G_n(u_{0,n},u_n;G_0)$ to \eqref{1}. Therefore, by Theorem  \ref{thm2}.3 we obtain
\begin{align}\label{28}
&|f(u_{0,n},u_n)|=\nonumber\\
&=\left|\int_0^{1}\int_0^{T}\left(\frac{\beta}{2}u_n^2(t,a)+\frac{\beta}{2}u_{0,n}^2(t)-K_\Pi G^\gamma_n(t,a)\right) dtda\right|\nonumber\\&\geq
\frac{\beta}{2}\left(\|u_n\|^2_{L^2((0,T)\times(0,1))}+\|u_{0,n}\|^2_{L^2(0,T)}\right)-K_\Pi\int_0^{1}\int_0^{T}| G_n(t,a)|^\gamma dtda\nonumber\\&\geq\frac{\beta}{2}\left(\|u_n\|^2_{L^2((0,T)\times(0,1))}+\|u_{0,n}\|^2_{L^2(0,T)}\right)-K_\Pi\int_0^{T}\left(\int_0^{1}| G_n(t,a)|^2 da\right)^{\frac{\gamma}{2}}dt\\&\geq
\frac{\beta}{2}\left(\|u_n\|^2_{L^2((0,T)\times(0,1))}+\|u_{0,n}\|^2_{L^2(0,T)}\right)-K_\Pi\sup_{t\in[0,T]}\| G_n(t)\|_{L^2(0,1)}^{\gamma}\nonumber\\&\geq
\frac{\beta}{2}\left(\|u_n\|^2_{L^2((0,T)\times(0,1))}+\|u_{0,n}\|^2_{L^2(0,T)}\right)-K_\Pi L_1^\gamma\left(\|G_0\|_{L^2(0,1)}+\|u\|^\rho_{L^2((0,1)\times(0,T))}+\|u_0\|^\rho_{L^2(0,T)}\right)^\gamma.\nonumber
\end{align}
Thus, since $\rho,\gamma\in(0,1]$, $|f(u_n,v_n)|\to \infty$ as $n$ tends to $\infty$. Hence the functional $f$ is coercive.

Since for arbitrary $(u,v)\in U_{ad}\times U_{0,ad}$ a generalised mild solution $G=G(u,v)\in C(0,T;L^2(0,1))$ to \eqref{1} continuously depends on $(u,v)\in H$, the functional $J:H\to \R$ is continuous.  

Therefore, from Theorem \ref{thm4} there exists a unique optimal solution to control problem \eqref{1}, \eqref{J}.
\endproof
\subsection{Necessary optimality conditions}
In this section we assume that \eqref{expstability} hold and $G_0\in L^{\infty}(0,1)$ is  positive a.e. Then by Remark \ref{r:4} and Theorem \ref{thm5} there exists an unique optimal solution $(G^*,u_0^*,u^*)\in L^{\infty}(0,T;L^{\infty}(0,1))\cap C(0,T;L^2(0,1))\times U_{0,ad}\times U_{ad}$ to the problem of maximizing \eqref{J} subject to \eqref{1}. Furthermore, 
from Proposition 2 in \cite{Feichtinger2003} it follows that there exists a unique solution $\xi\colon\left[0,T\right]\times\left[0,1\right]\rightarrow \mathbb{R}$ to the adjoint system: 
\begin{align}\label{4}\left\{\begin{array}{lc}
\frac{\partial \xi(t,a)}{\partial t}+\frac{\partial \xi(t,a)}{\partial a}=K_\Pi e^{-rt}\gamma (G^*(t,a))^{\gamma-1}+\xi(t,a)\delta(a)- \xi(t,0)R(a)\medskip & (t,a)\in\left[0,T\right]\times\left[0,1\right],\\
\xi(T,a)=0& a\in\left[0,1\right], \\
\xi(t,1)=0 & t\in\left[0,T\right].\end{array}\right. 
\end{align}
We follow \cite{Feichtinger2003} in defining  
the Hamiltonian associated with boundary condition 
\begin{align*}
H_b(t,u_0)=&\xi(t,0)\left(\int_0^{1}\left(R(a)G^*(t,a)+u^\rho(t,a)\right)da+u_0^\rho\right)\\&-\int_0^{1}e^{-rt}\left(K_\Pi (G^*(t,a))^\gamma-\frac{\beta}{2}(u^*(t,a))^2\right)da+e^{-rt}\frac{\beta}{2}u_{0}^2,
\end{align*}
for a.e. $t\in[0,T]$ and every $u_0\in[0,I]$, and the distributed Hamiltonian takes the form
\begin{align*}
H(t,a,u)=&e^{-rt}\left(\frac{\beta}{2}u^2+\frac{\beta}{2}(u^*_0)^2(t)+C_f-K_\Pi (G^*(t,a))^\gamma\right)+\xi(t,a)\left(-\delta(a)G^*(t,a)+u^\rho\right)\\&+\xi(t,0)\left(R(a)G^*(t,a)+u^\rho\right).\end{align*}
for a.e.  $(t,a)\in[0,T]\times[0,1]$ and every $u\in[0,I]$.

Based on maximum principle introduced in \cite{Feichtinger2003}  the optimal solution for the problem \eqref{J} with \eqref{1} satisfies 
 
\begin{align}\label{v}
u_0^*(t)=\left\{\begin{array}{ll}
0 & \textrm{ for } \xi(t,0)> 0\\
\left(- \frac{\rho}{\beta}e^{rt}\xi(t,0)\right)^\inv{2-\rho}& \textrm{ for } \xi(t,0)\in[0,-\frac{\beta}{\rho}e^{-rt}I^{2-\rho}]\\
I & \textrm{ for } \xi(t,0)<-\frac{\beta}{\rho}e^{-rt}I^{2-\rho}
\end{array}\right.,
\end{align}
 and
\begin{align}\label{u}
u^*(t,a)=\left\{\begin{array}{ll}
0 & \textrm{ for } \xi(t,0)+\xi(t,a)> 0\\
\left(-\frac{\rho}{\beta}e^{rt}\left(\xi(t,a)+\xi(t,0)\right)\right)^{\inv{2-\rho}}& \textrm{ for } \xi(t,0)+\xi(t,a)\in[0,-\frac{\beta}{\rho}e^{-rt}I^{2-\rho}]\\
I & \textrm{ for } \xi(t,0)+\xi(t,a)<-\frac{\beta}{\rho}e^{-rt}I^{2-\rho}
\end{array}\right..
\end{align}
 for a.e. $(t,a)\in[0,T]\times[0,1]$.

Summarizing the above considerations, the optimal triple $(G^*,u_0^*,u^*)$ is the solution of the following system
\begin{align}\label{5}
\left\{\begin{array}{lc}
\frac{\partial G^*(t,a)}{\partial t}+\frac{\partial G^*(t,a)}{\partial a}+\delta(a) G^*(t,a)=(u^*(t,a))^\rho\medskip & (t,a)\in\left[0,T\right]\times\left[0,1\right],\\
G^*(t,0)=\int_0^{1}\left(R(a)G^*(t,a)+(u^*(t,a))^{\rho}\right)da+(u^*_0(t))^{\rho}& t\in\left[0,T\right], \\
G^*(0,a)=G_0(a) & a\in\left[0,1\right]\\
\frac{\partial \xi(t,a)}{\partial t}+\frac{\partial \xi(t,a)}{\partial a}=K_\Pi e^{-rt}\gamma (G^*(t,a))^{\gamma-1}+\xi(t,a)\delta(a)- \xi(t,0)R(a)\medskip & (t,a)\in\left[0,T\right]\times\left[0,1\right],\\
\xi(T,a)=0& a\in\left[0,1\right], \\
\xi(t,1)=0 & t\in\left[0,T\right].
\end{array}\right. 
\end{align} 
where $u_0^*, u^*$ are given by \eqref{v} and \eqref{u}.
\section{Numerical solution} 

\label{s:numerical}
The system of equations \eqref{5} does not possess an explicit solution. Therefore, in order to analyse the properties of optimal trajectories, we solve \eqref{5} numerically. For this purpose we apply well-known approach in the numerical analysis of PDEs - so-called method of lines (MOL) (cf. \cite{kamont1999}, \cite{Schiesser2009}).
In the first step we use a finite difference approximation to discretize the space variable $a$ on a selected space mesh.  
Thus, let $\{ a_0=0,a_1,a_2,\ldots,a_N=1\}$ be uniform grid of the consumers' segments and $\Delta a=\Delta a_i=a_i-a_{i-1}$ be the diameter of this division. In the segment $a_i$ for $i=0,1,\ldots,N$ we denote: $G_i^*(t)=G^* (t,a_i)$, $\xi_i (t)=\xi(t,a_i)$, $R_i=R(a_i)$, $\delta_i=\delta(a_i)$, $u_i^*(t)=u^*(t,a_i )$.

Moreover, we apply the composite trapezoidal rule for the approximation of the definite integral \cite[p.153]{gautschi1997}
$$\int_0^{1}\left(R(a)G^*(t,a)+(u^*(t,a))^{\rho}\right)da \approx \Delta a\left(\frac{1}{2}f_1(t)+\sum_{i=2}^{N-1}f_i(t)+\frac{1}{2}f_N(t)\right),$$ where $$f_i(t)=R_iG_i^*(t)+\left(u_i^*(t)\right)^{\rho}\quad \textrm{for}\quad i=1,\ldots,N,$$  
and the explicit and the implicit Euler schemes as the approximations of derivatives:
\begin{align*}
\frac{\partial G^*(t,a_i )}{\partial a}=\frac{G_i^* (t)-G_{i-1}^* (t)}{\Delta a},\quad  \textrm{ dla } i=1,2,\ldots,N,\\
\frac{\partial \xi(t,a_i )}{\partial a}=\frac{\xi_{i+1}(t)-\xi_i (t)}{\Delta a},\quad \textrm{ dla } i=0,1,\ldots,N-1.\end{align*}
Therefore, the system \eqref{5} is transformed to the system of $2N$ ordinary differential equations and the resulting system becomes
\begin{align}\label{6}
\left\{\begin{array}{ll}
\frac{dG_i^* (t)}{dt}=-G_i^* (t)\left(\delta_i+\frac{1}{\Delta a}\right)+ \frac{1}{\Delta a}G_{i-1}^* (t)+\left(u_i^* (t)\right)^{\rho}, & t\in(0,T],\ i=1,\ldots,N,\medskip\\
G_0^* (t)=\Delta a\left(\frac{1}{2}f_1(t)+\sum\limits_{i=2}^{N-1}f_i(t)+\frac{1}{2}f_N(t)\right)+\left(u_0^* (t)\right)^{\rho}, &  t\in(0,T],\\
G_i^* (0)=G^*_{0,i}, & i=0,\ldots,N,\\
\frac{d\xi_i (t)}{dt}=K_\Pi e^{-rt}\gamma (G_i^* (t))^{\gamma-1}-\xi_0(t)R_i+\xi_i (t)\left(\delta_i +\frac{1}{\Delta a}\right)-\frac{1}{\Delta a}\xi_{i+1}(t),& t\in (0,T], i=1,\ldots,N-1,\medskip\\
\xi_N(t)=0,& t\in[0,T),\\
\xi_i(T)=0,& i=0,\ldots,N-1,\\
f_i(t)=R_iG_i^*(t)+\left(u_i^*(t)\right)^{\rho} & i=0,\ldots,N\\
\end{array}
\right.
\end{align}
with the controls $u_i^*$ of the forms
\begin{equation*}
u_0^*(t)=\left\{\begin{array}{ll}
0 & \textrm{ for } \xi_0(t)> 0\\
\left(- \frac{\rho}{\beta}e^{rt}\xi_0(t)\right)^\inv{2-\rho}& \textrm{ for } \xi_0(t)\in\left[0,-\frac{\beta}{\rho}e^{-rt}I^{2-\rho}\right]\\
I & \textrm{ for } \xi_0(t)<-\frac{\beta}{\rho}e^{-rt}I^{2-\rho}
\end{array}\right.,
\end{equation*}
 and for $i=1,\ldots,N$

\begin{equation*}
u_i^*(t)=\left\{\begin{array}{ll}
0 & \textrm{ for } \xi_0(t)+\xi_i(t)> 0\\
\left(-\frac{\rho}{\beta}e^{rt}\left(\xi_i(t)+\xi_0(t)\right)\right)^{\inv{2-\rho}}& \textrm{ for } \xi_0(t)+\xi_i(t)\in\left[0,-\frac{\beta}{\rho}e^{-rt}I^{2-\rho}\right]\\
I & \textrm{ for } \xi_0(t)+\xi_i(t)<-\frac{\beta}{\rho}e^{-rt}I^{2-\rho}
\end{array}\right..
\end{equation*}
The system \eqref{6} can be posed a non-linear boundary value problem (BVP) and it is solved with the Matlab solver bvp5c. 

In this approach, one can encounter two main difficulties. The first complication is the choice of an adequate number of spatial grid. The grid containing a very large number of nodes causes the MOL approximation \eqref{6}  is close to the system \eqref{5}. However, it increases the number  of ordinary differential equations in \eqref{6}, thereby increasing time to solution and reducing stability and accuracy of solution to \eqref{6}. The former is closely connected to maximal ratio of the time step $\Delta t$ and the space step $\Delta a$ (i.e. Courant-Friedricks-Lewy number) which should be small enough.

The second difficulty occur with implementation of the Matlab solver bvp5c. For BVP solutions the most difficult part is providing an initial approximation to the solution i.e. guess function such that bvp5c solver leads to convergence (see \cite{shampine2003}). 
Therefore, we apply the iterative procedure for solving the system \eqref{6}. At the beginning, we establish a very sparse mesh for space division and based on the polynomial interpolation with respect to the initial condition in \eqref{6} we obtain first guess function and then, using bvp5c procedure we find the initial solution to \eqref{6}. Next, we increase the division of space and find a new guess function based on polynomial interpolation of the initial solution to \eqref{6} and using it we solve the system \eqref{6} again. This scheme is repeated until we obtain a sufficient degree of accuracy of the solutions.

The analysis of convergence of some numerical MOL schemes can be found in \cite{verwer1984} \cite{kamont1999}, but the convergence of the solutions of \eqref{6} to the solution of PDEs \eqref{5} is still an open problem.

\section{Simulation of the goodwill model}
 \label{s:simulation}

In this section, using the results from  Sections \ref{s:optimal_sol}--\ref{s:numerical}, we will find the optimal advertising strategies and the corresponding optimal trajectories of goodwill. In particular, we will draw attention to the impact of different values of the goodwill elasticity of demand ($\epsilon_g$) and the parameter of the advertising response function ($\rho$) on the optimal solution and the level of the firm's profit. We consider a durable experience product. Hence the consumers do not purchase the product frequently and they learn about the attributes of the product after using it for some time (see \cite{nelson1974}). Moreover, we assume the product is low quality and the longer consumers use  the product, the lower the proportion of them evaluate it positively.  Therefore, we assume that the rate of consumer recommendation $R$ is decreasing with respect to $a$ and takes the form $R(a)=\frac{3}{5}-\frac{3}{21}\sqrt{a}$. Furthermore, we take an a increasing  goodwill depreciation rate: $\delta(a)=1-\frac{0.5}{1-e^{-1}}e^{-a}$, which reflects the fact that as time goes by, more and more customers might become disappointed about the product's functionality. 

Different values of the goodwill elasticity  of demand $\epsilon_g$ are related to the consumer response to advertising. A low goodwill elasticity of demand may occurr in a situation where the consumer has commitments which block the use of a substitute product. Therefore, despite the fact that advertising has convinced the consumer to use another product, the consumer is only  able to purchase  an additional part of the service. A good example is that of mobile operators, where  post-paid service requires the consumers to sign a contract. Thus, their contribution to demand is relatively small. By contrast, pre-paid customers do not have any obligation to the mobile operator and may change firms at any time. Therefore, the goodwill elasticity of demand for this group of users is high.

In addition, we examine how the non-linear shape of the advertising response function in the goodwill equation affects the optimal advertising strategies, optimal goodwill path, and  the firm's profit. For these reasons, we analyse the linear $\rho=1$ and concave-downward $\rho=0.5$ advertising responses.

Besides, we assume that the rate of interest $r$ is equal to 2.8\%, the length of the product life cycle $T=1$, the unit advertising cost $\beta=0.16$, and the parameter $K=0.34$. The initial level of goodwill is $G_0(a)=1.5$. 

The results of the simulations are shown in Figures \ref{f3}--\ref{f2}. Each graphical presentation  consist of the following four plots (from left to right): contour plot of the optimal advertising strategy, 3D plot of the optimal advertising strategy, contour plot of the optimal goodwill path, 3D plot of the optimal goodwill path. 

We find two types of optimal advertising strategies: we will refer to them as `supportive' and `strengthening'. The first maintains the level of goodwill at most at its initial level, while the latter causes a significant increase in the level of goodwill from its initial value. Different types of optimal advertising strategies are responsible for different shapes of the associated optimal goodwill paths.  The optimal paths of product goodwill $G^*$ for a strengthening strategy reaches a maximum value for segments of consumers with short experience and these values spread over time to the segments of consumers with longer experience. Whereas for the second type, the maximum value of $G^*$ is achieved in all segments at the beginning of the product life cycle and then decreases, and the rate of decline is greatest among consumers with a short usage experience.

The `supportive' strategies are found in scenarios with low goodwill elasticity of demand. Moreover, one may observe two shapes of these optimal advertising strategies in each market segment $a$: decreasing concave (Figure \eqref{f3}) and parabolic with a maximum  (Figure \eqref{f4}). 

\begin{figure}[h!]
\begin{center}$
\begin{array}{cccc} \includegraphics[scale=0.2]{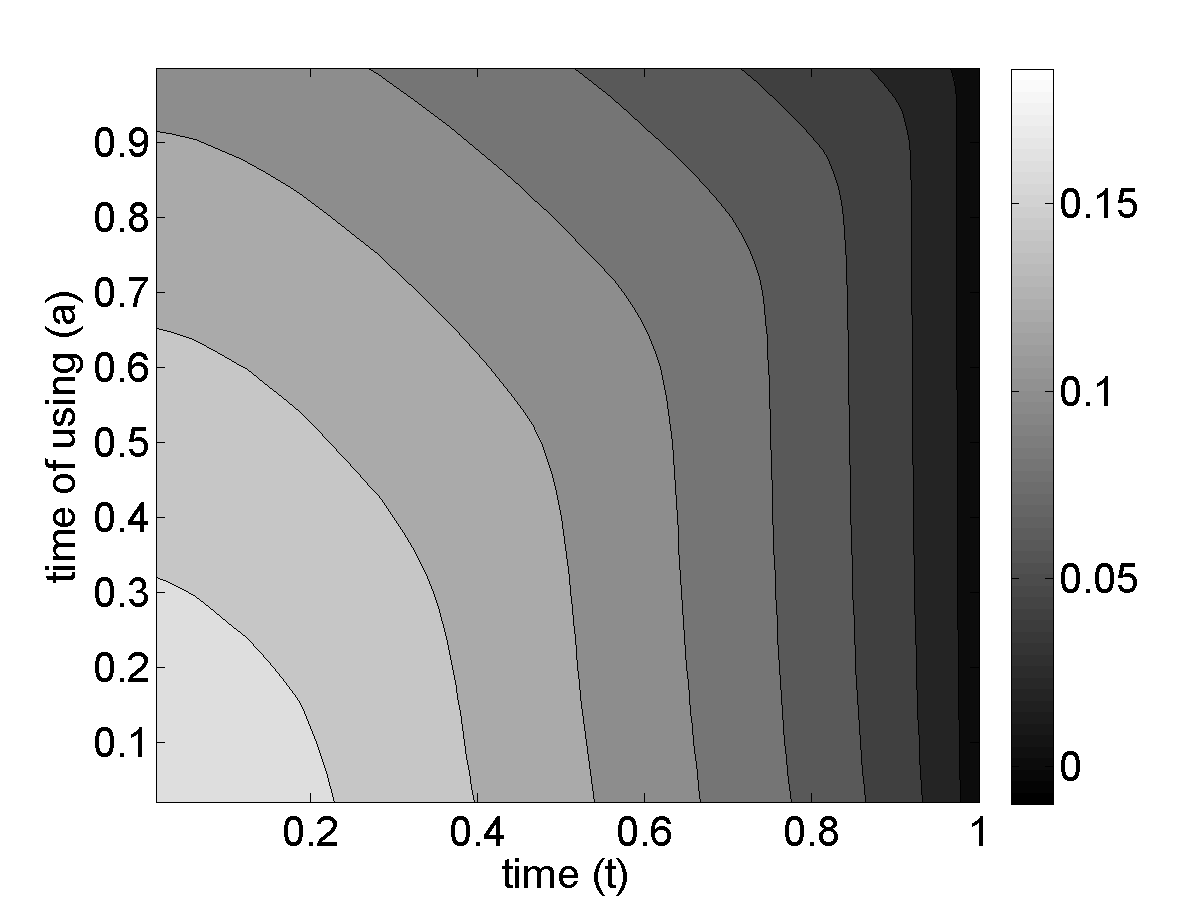} & \includegraphics[scale=0.2]{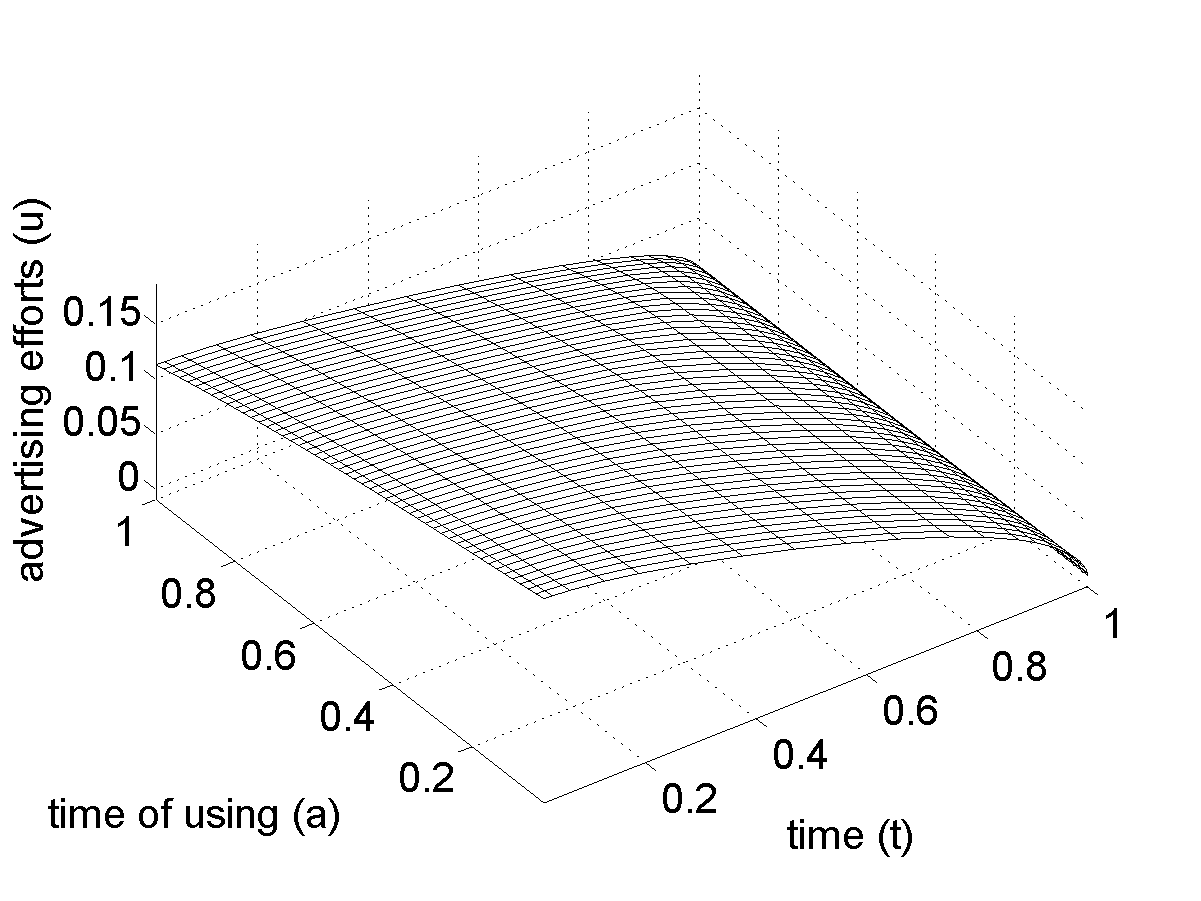} &
\includegraphics[scale=0.2]{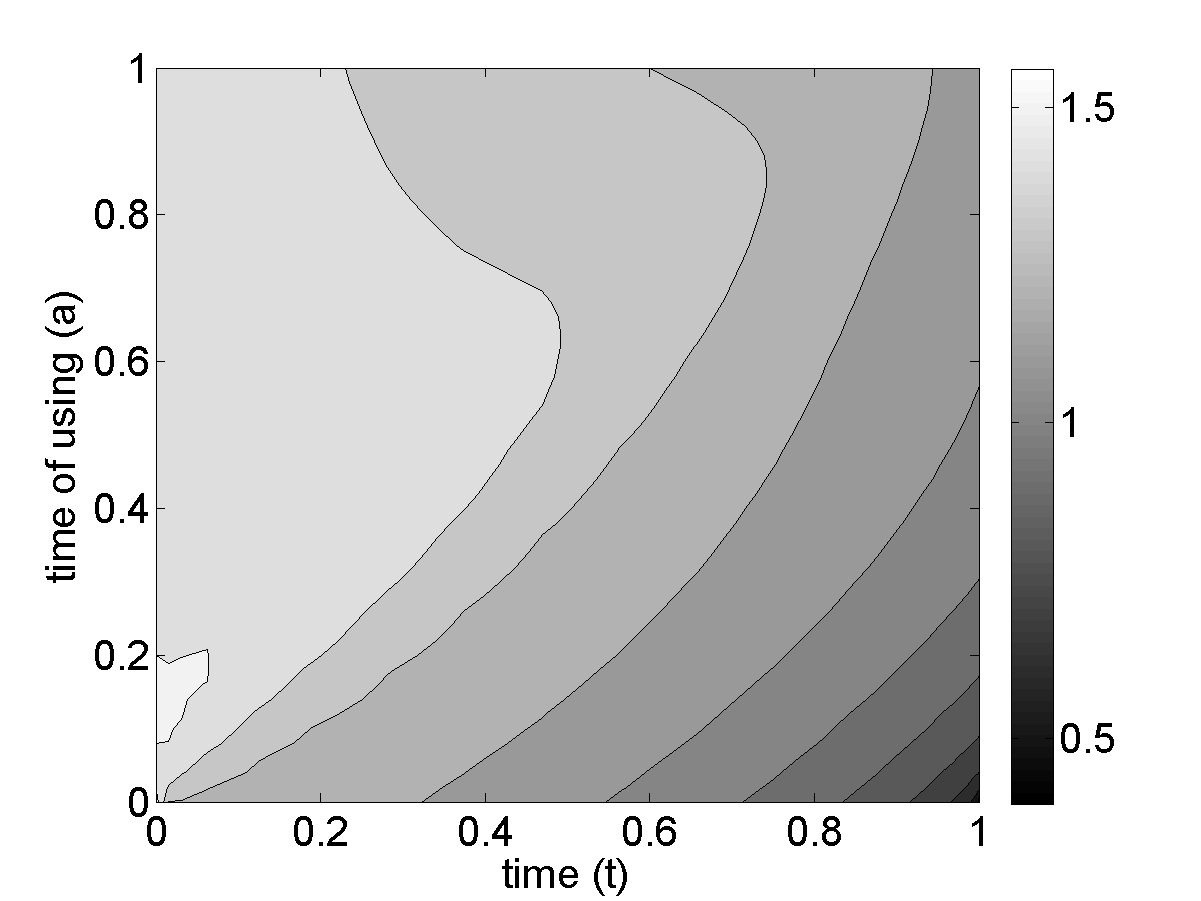} & \includegraphics[scale=0.2]{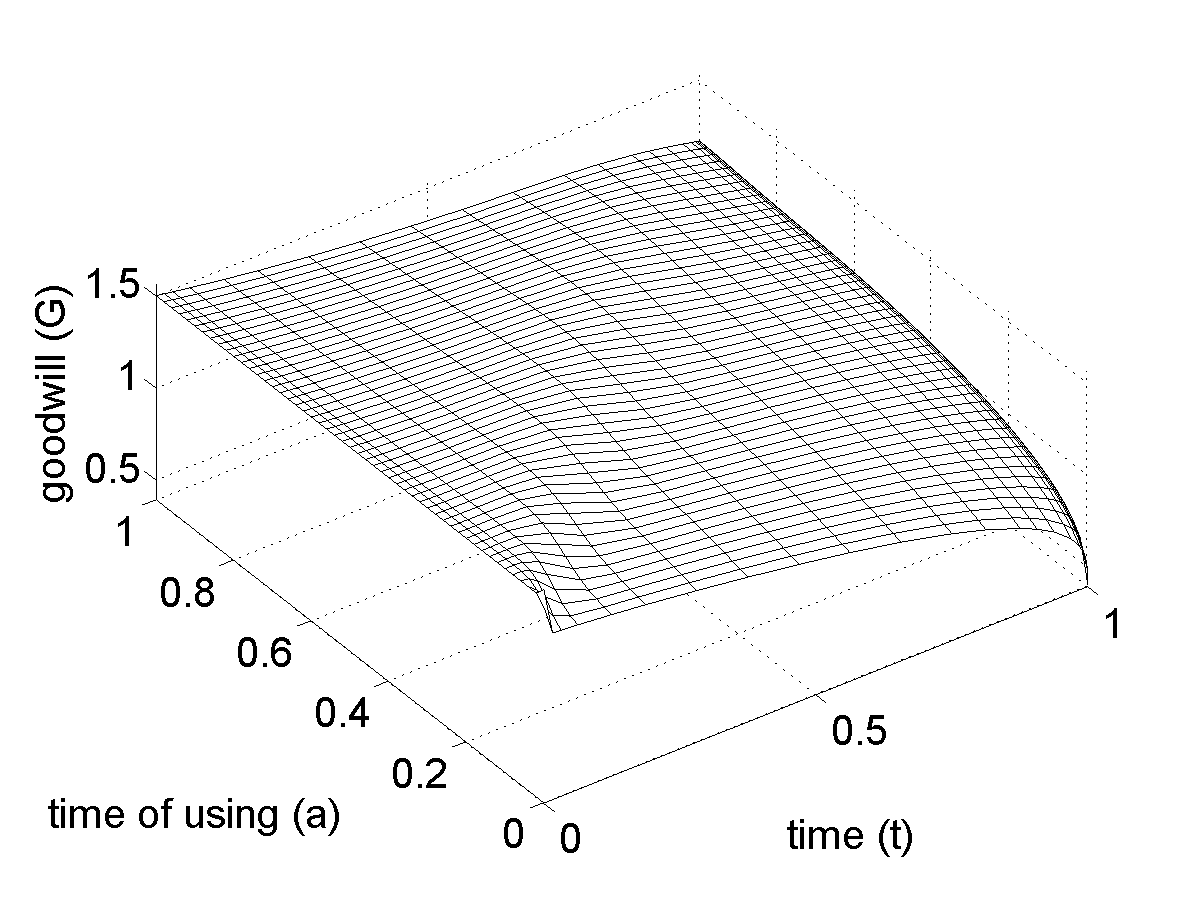}
\end{array}$
\end{center}
\caption{\textit{The optimal advertising strategy $u^*$ and optimal goodwill path $G^*$ in the experiment for $\epsilon_g=0.1$; $\rho=0.5$.}}
\label{f3}
\end{figure}

\begin{figure}[h!]
\begin{center}$
\begin{array}{cccc} \includegraphics[scale=0.2]{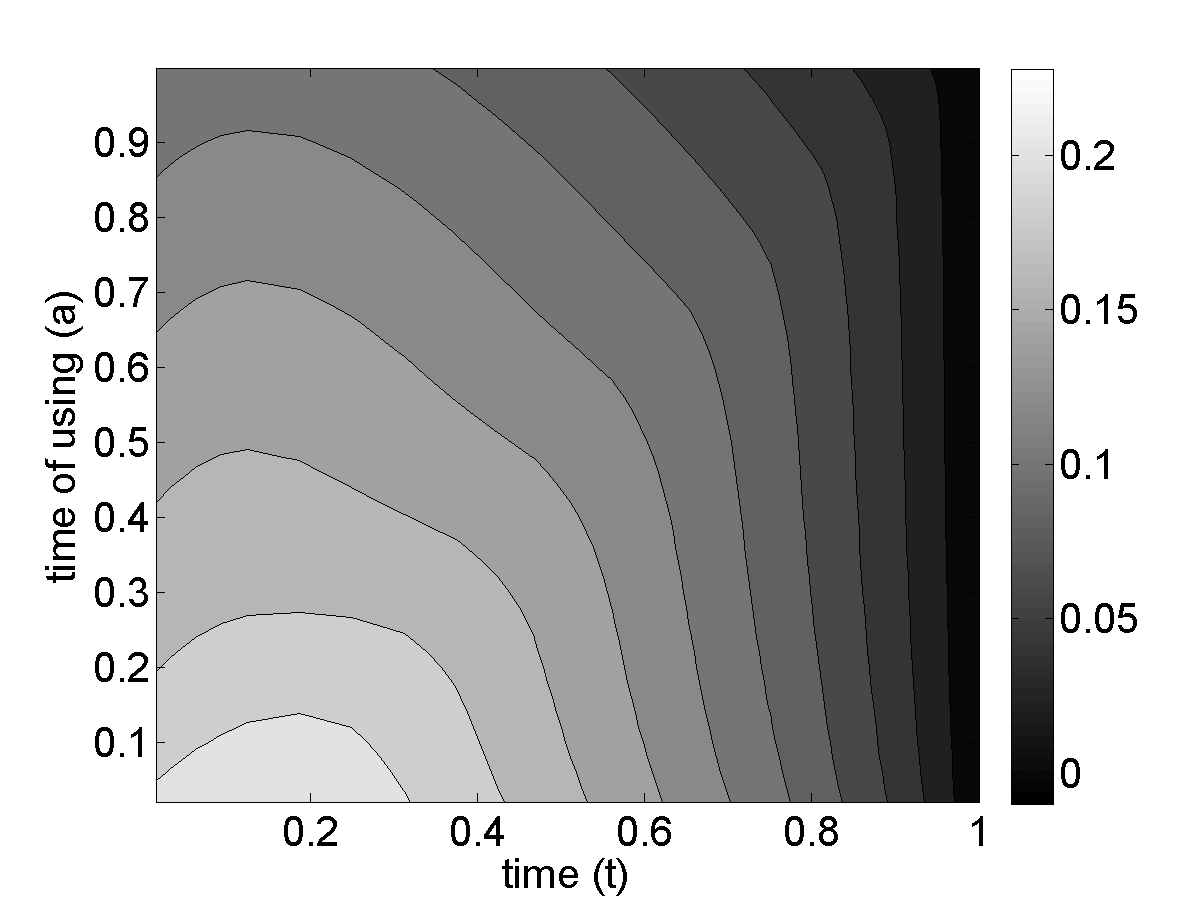} & \includegraphics[scale=0.2]{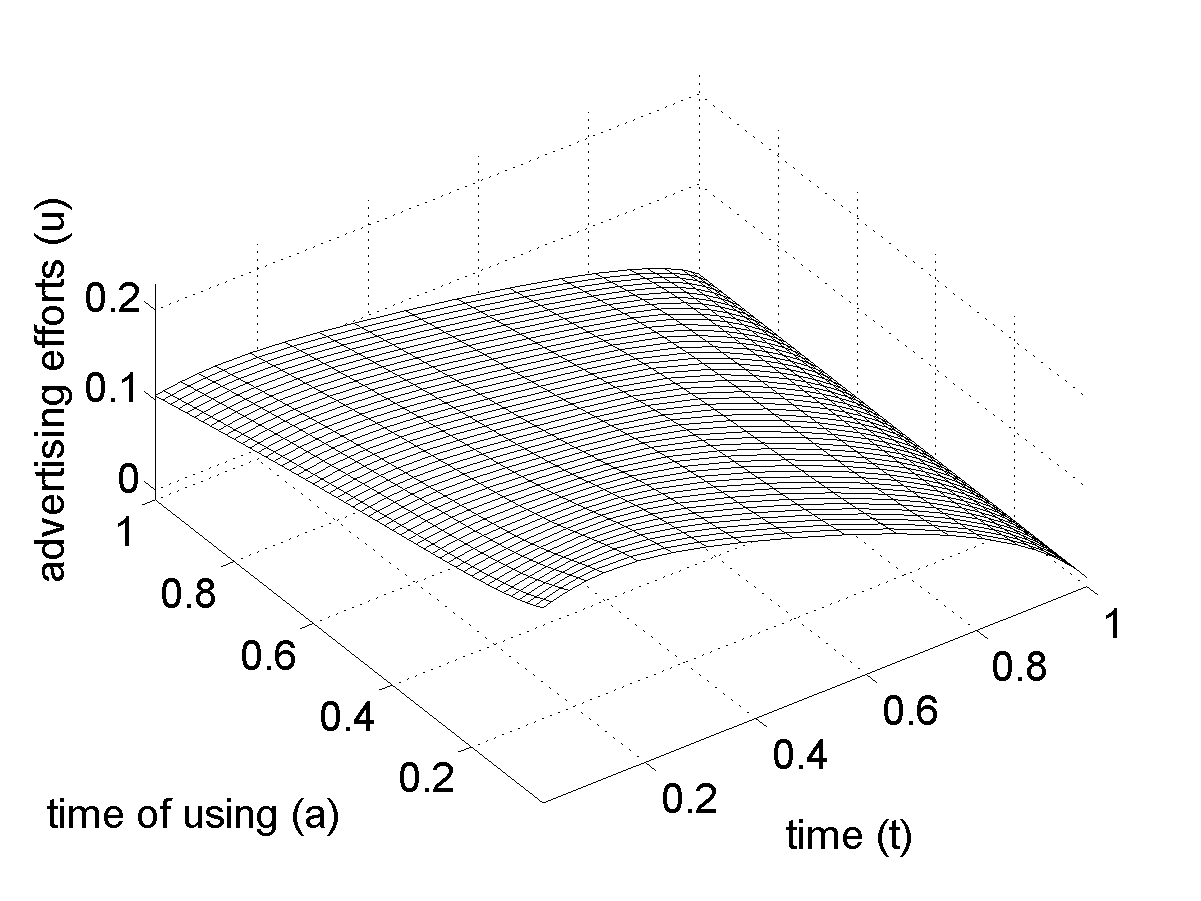} &
\includegraphics[scale=0.2]{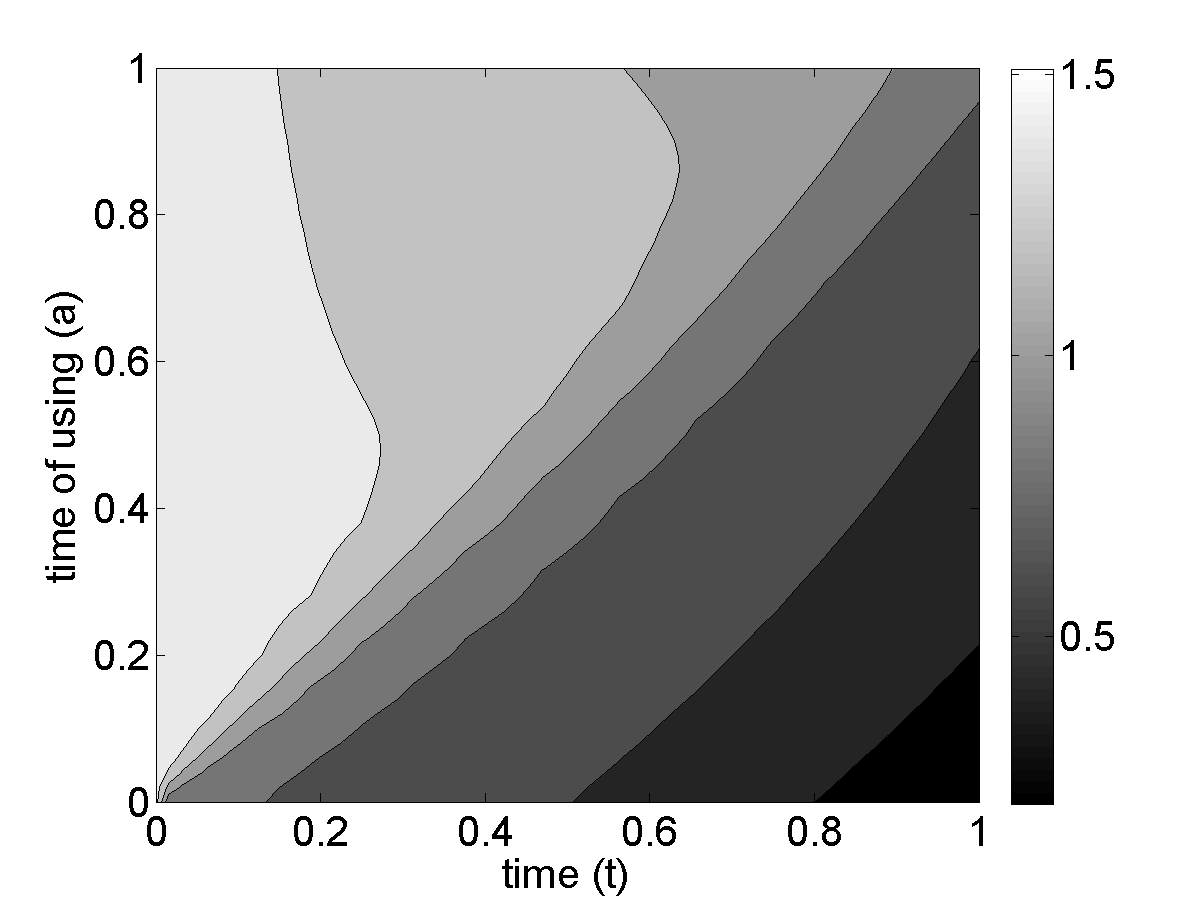} & \includegraphics[scale=0.2]{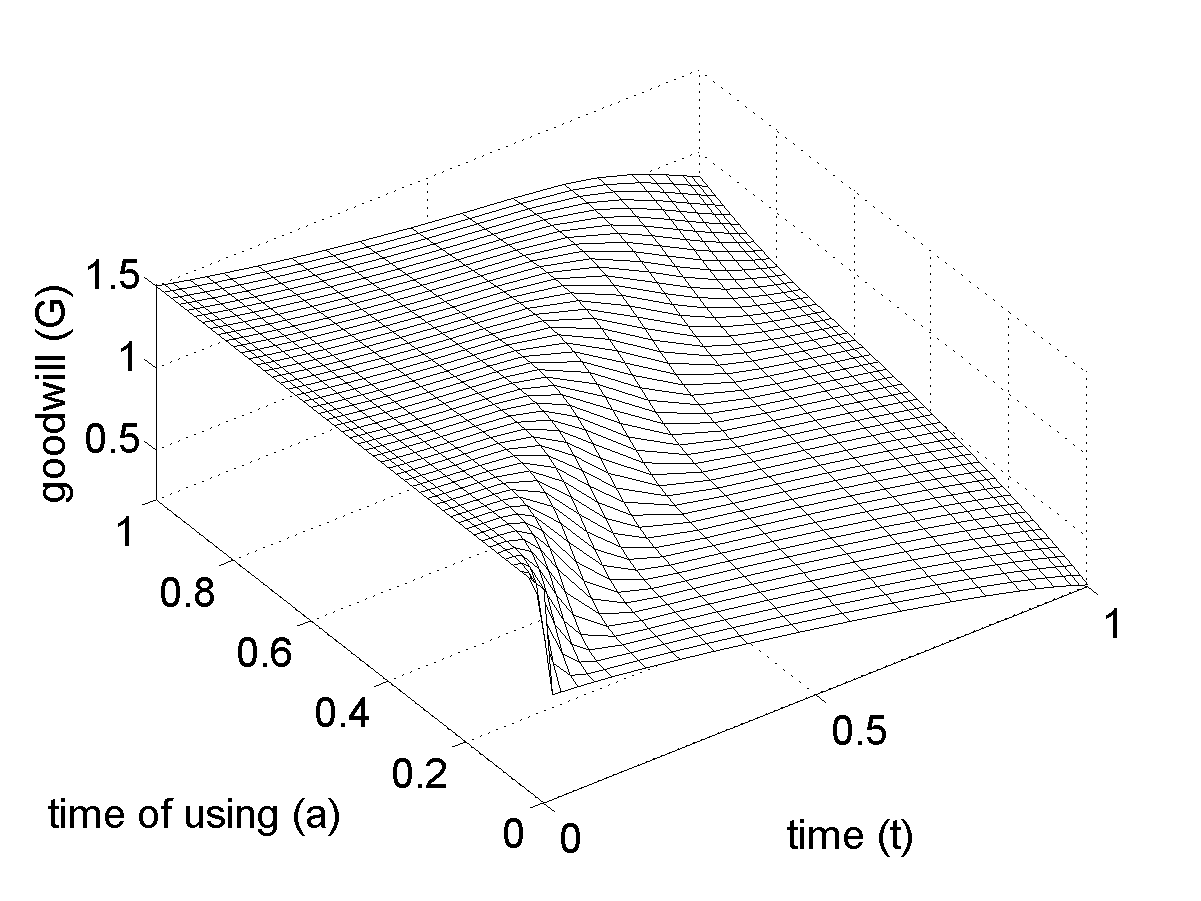}
\end{array}$
\end{center}
\caption{\textit{The optimal advertising strategy $u^*$ and optimal goodwill path $G^*$ in the experiment for $\epsilon_g=0.1$ and $\rho=1$.}}
\label{f4}
\end{figure}

A decreasing concave strategy can be found for $\rho=0.5$, and the parabolic with a maximum for $\rho=1$. This implies that including a non-linear advertising response function in the goodwill model results in the maximum level of $u^*$ occurring much later.  Moreover, the maximum level of optimal strategies with decreasing concave shape is smaller by 24\% than for the strategy with a parabolic shape. In a further part of this section, we will explore the financial consequences of these findings for the company. 

The second type of optimal advertising strategies are the strengthening strategies and they occur in our experiments for high values of the goodwill elasticity of demand and for  both values of the parameter $\rho$ (see Figures \eqref{f1} and \eqref{f2}). 

\begin{figure}[h!]
\begin{center}$
\begin{array}{cccc} \includegraphics[scale=0.2]{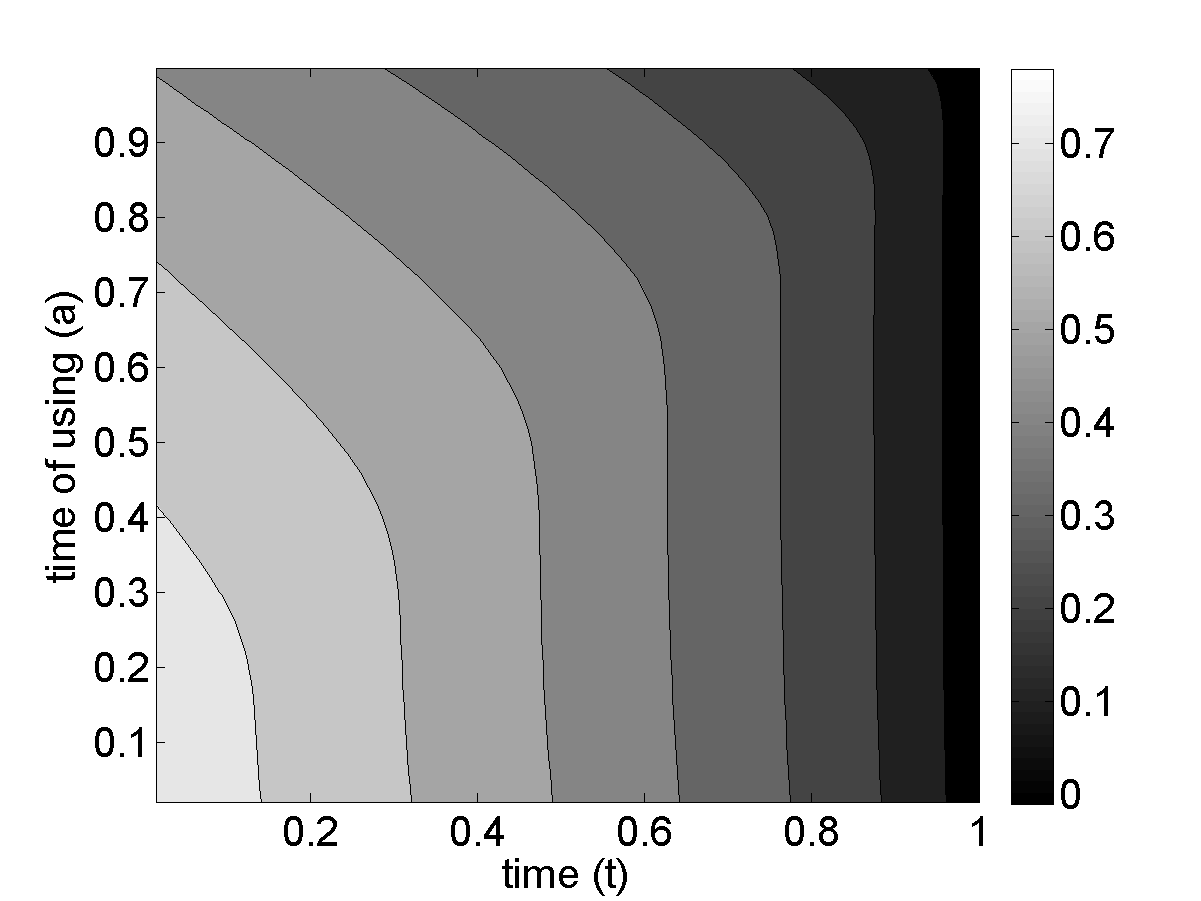} & \includegraphics[scale=0.2]{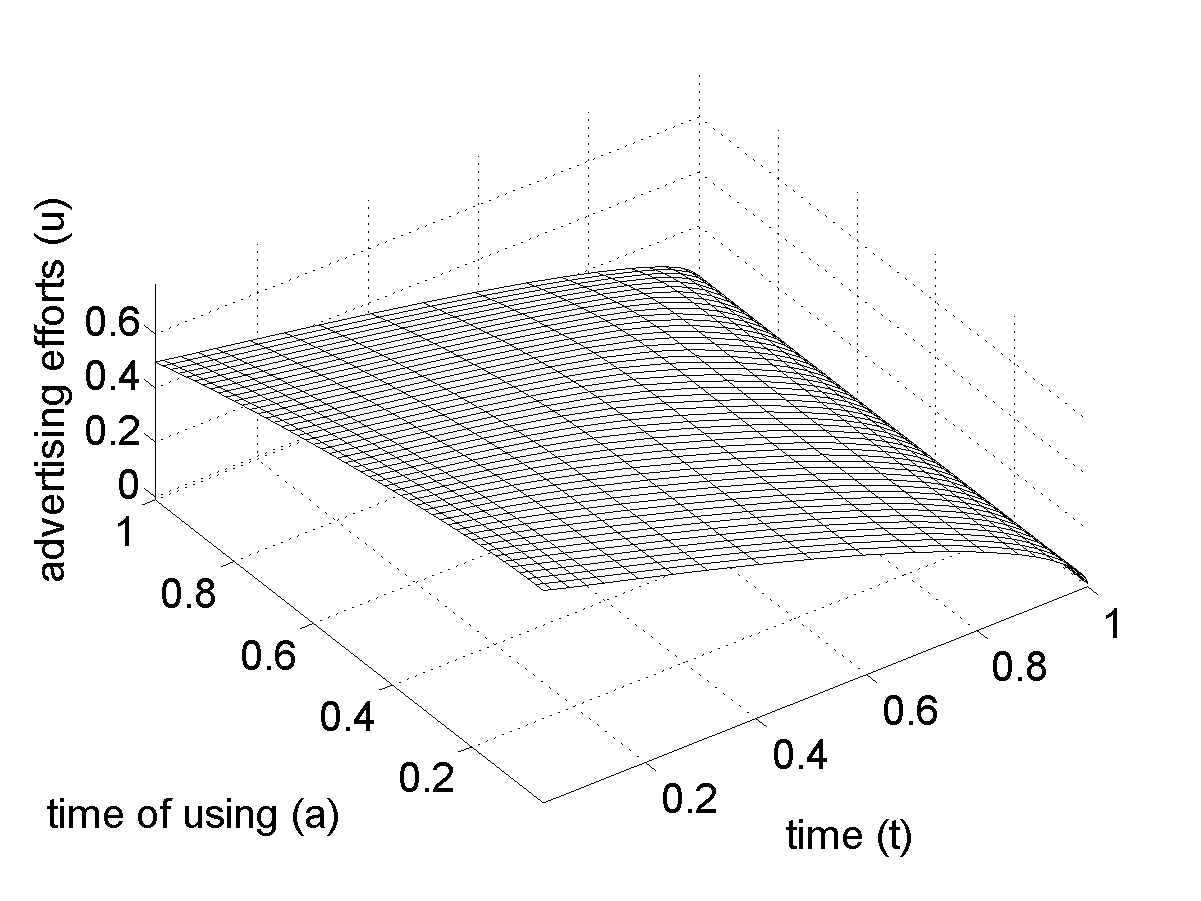} &
\includegraphics[scale=0.2]{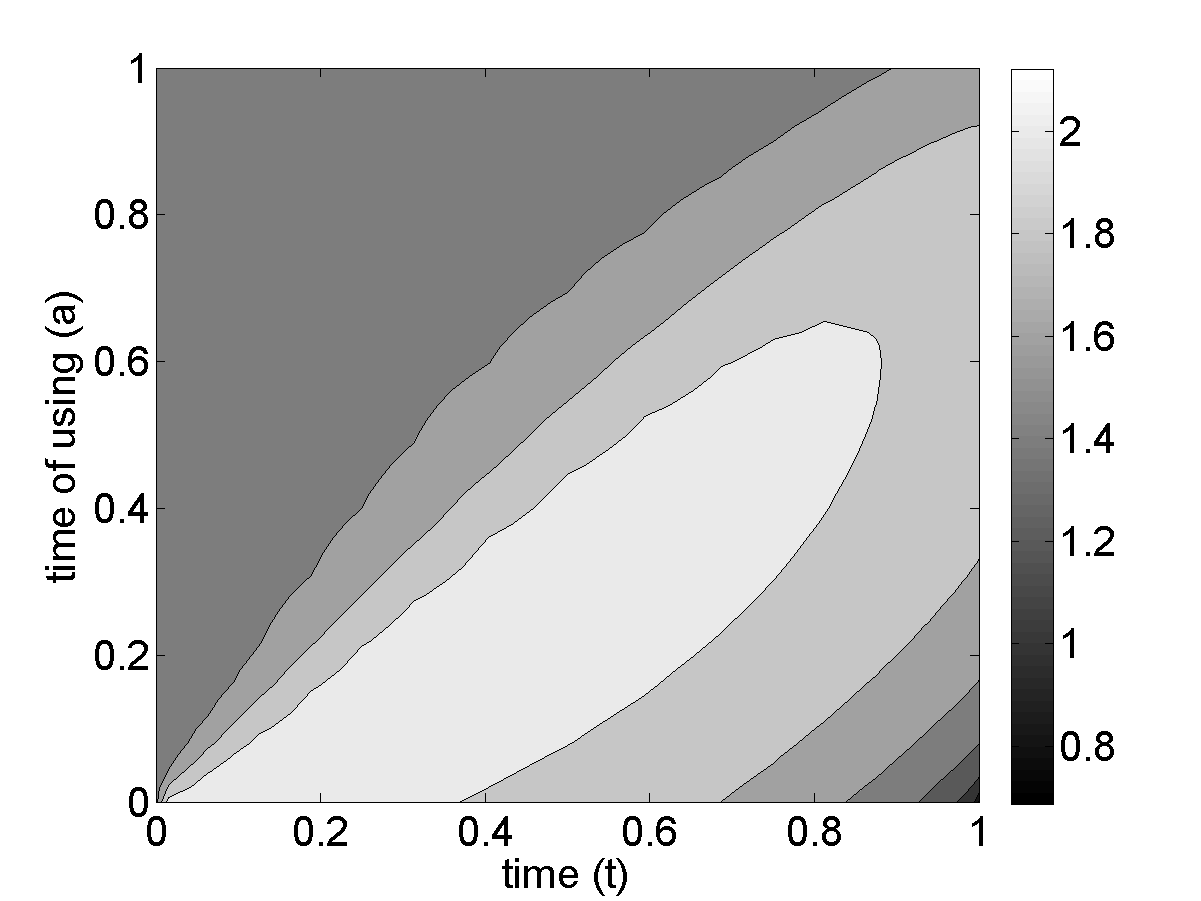} & \includegraphics[scale=0.2]{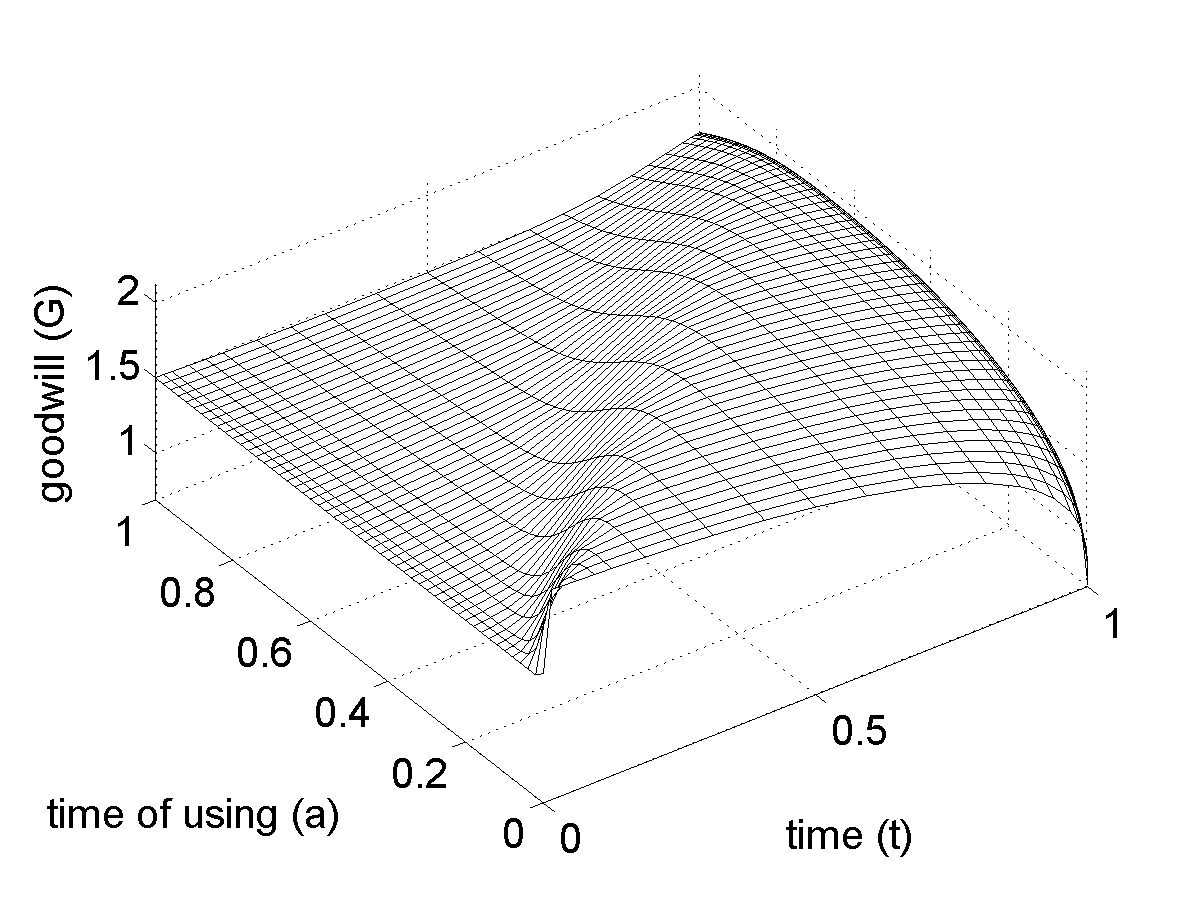}
\end{array}$
\end{center}
\caption{\textit{The optimal advertising strategy $u^*$ and optimal goodwill path $G^*$ in the experiment for $\epsilon_g=1$ and $\rho=0.5$.}}
\label{f1}
\end{figure}
\begin{figure}[h!]
\begin{center}$
\begin{array}{cccc} \includegraphics[scale=0.2]{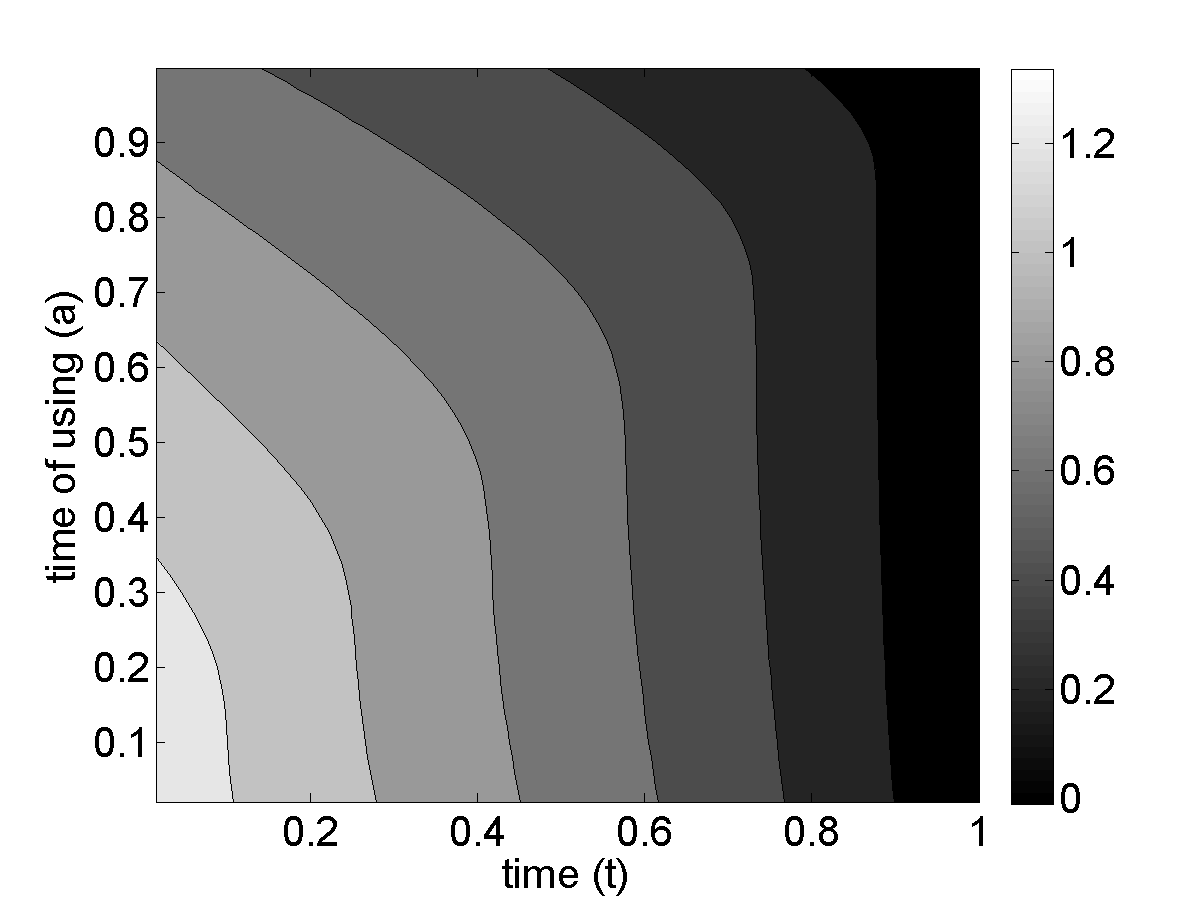} & \includegraphics[scale=0.2]{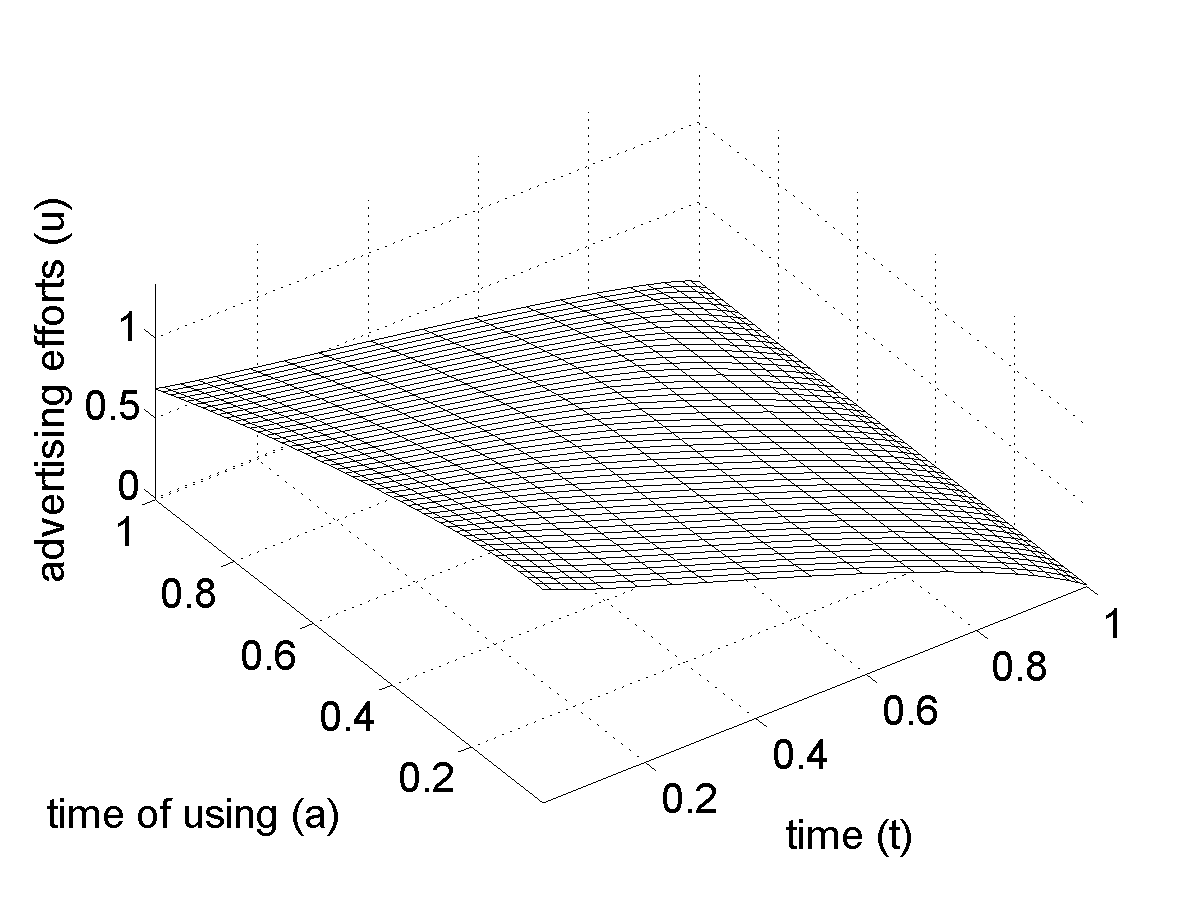} &
\includegraphics[scale=0.2]{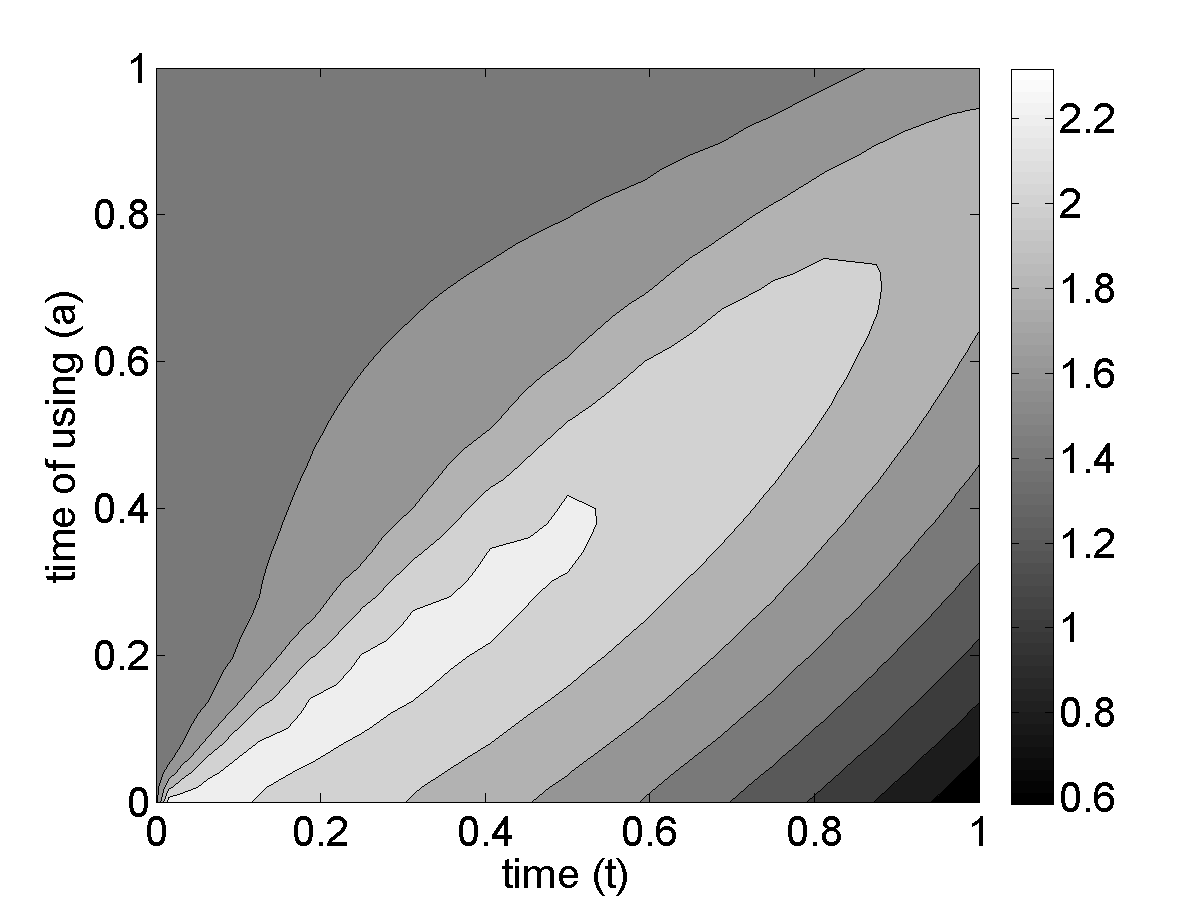} & \includegraphics[scale=0.2]{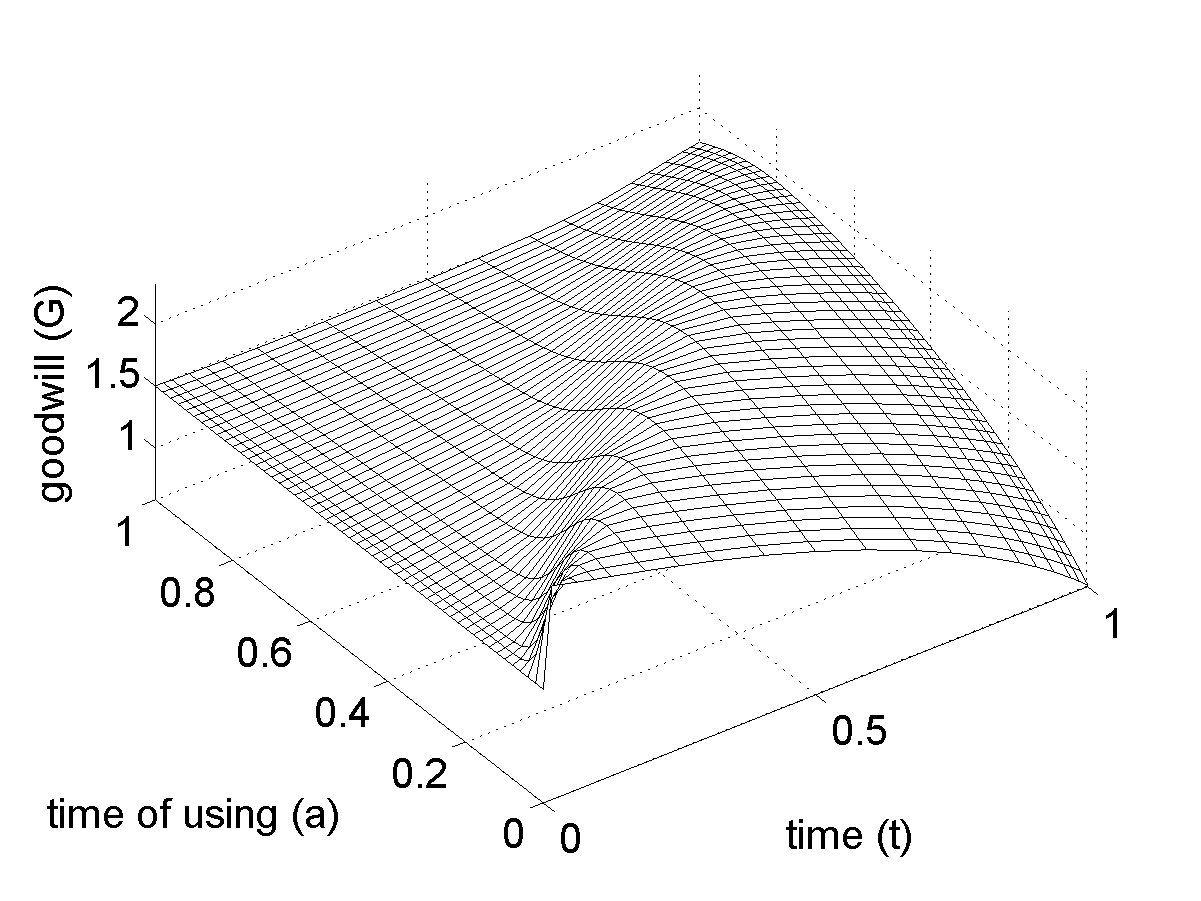}
\end{array}$
\end{center}
\caption{\textit{The optimal advertising strategy $u^*$ and optimal goodwill path $G^*$ in the experiment for $\epsilon_g=1$ and $\rho=1$.}}
\label{f2}
\end{figure}	

In both cases, the strengthening strategies have a concave decreasing shape in each market segment. This means that the maximum level of $u^*$ is reached at the beginning of the product life cycle and for consumers with shorter usage experience. However, there is a difference between the experiments with $\rho=0.5$ and $\rho=1$ in the maximum level of the optimal advertising strategies.  The greatest value of $u^*$ for the non-linear model is nearly one-half that for the linear  one (see Table \ref{tb:lq}), also resulting in a significant increase in the value of the total profit.

Finally, we compare several essential values obtained for the four scenarios with different model parameters $\rho$ and $\epsilon_g$. They are presented in Table \ref{tb:lq}. 

\begin{table}[h]
\centering
\begin{tabular}{|l|l|l|l|l|l|l|l|}\hline
$\rho$	&	$\epsilon_g$	&	$J_0$ & $J$ & $\Delta J/J_0$	&	$\max u$	&	$\max u_{0}$	 &	$\max G$	\\ \hline
0.5	&	0.1	&	0.31	&	0.318	&	3\%	&	0.175	&	0.11	&	1.55	\\\hline	
0.5	&	1	&	0.276	&	0.387	&	40\%	&	0.769	&	0.485	&	2.111	\\\hline	
1	&	0.1	&	0.31	&	0.313	&	1\%	&	0.217	&	0.108	&	1.5	\\\hline	
1	&	1	&	0.276	&	0.36	&	30\%	&	1.325	&	0.662	&	2.307	\\\hline	
\end{tabular}
\caption{\textit{Experiments for goods with low quality. $J_0$ is the firm's profit without advertising investment}.}
\label{tb:lq}
\end{table}

 The ratio  $\frac{\Delta J}{J_0}$ is a measure of the benefit from an advertising campaign and is equal to the percentage change of the firm's profits caused by introducing the optimal advertising campaign. Our simulations confirm that a low level of goodwill elasticity of demand causes  a small percentage increase in the firm's profits $\frac{\Delta J}{J_0}$ and that the advertising intensities and goodwill paths have significantly lower magnitudes  than in the case where $\epsilon_g=1$.

The above analysis highlights the importance of consumer recommendations, the levels of goodwill elasticity, and a non-linear advertising response function in creating an optimal  advertising campaign, and should be taken into account by managers.

\section*{Acknowledgements}
The authors gratefully acknowledge financial support from the National Science Centre in Poland. Decision number: DEC-2011/03/D/HS4/04269.



\bibliographystyle{elsarticle-harv}
\bibliography{literaturagood_a}
\end{document}